\documentclass[12pt,draftcls,onecolumn]{IEEEtran-TR}

\usepackage{balance}
\usepackage{graphics}
\usepackage{graphicx}
\usepackage{amssymb}
\usepackage{verbatim}     
\usepackage{amsxtra}%
\usepackage{epsfig}
\usepackage{subfigure}
\usepackage{hyperref} 
\usepackage{amsmath}
\usepackage{latexsym}
\usepackage{ifthen}
\usepackage{psfrag}
\usepackage{subfigure}
\usepackage{color}

\usepackage{rgsEnvironments}
\usepackage{rgsMacros}

\graphicspath{{/Users/Ricardo/Dropbox/Graphics/}}
\newcommand{\Ir}{{\cal I}(r)}
\newcommand{\Irzero}{{\cal I}(0)}
\newcommand{\tb}{\bar{\tau}}
\newcommand{\ton}{\tau_1}
\newcommand{\ttw}{\tau_2}

\newcommand{\one}{\mathbf 1}
\renewcommand{\graph}[1]{{\mbox{gph}#1}}

\setboolean{Conference}{false}

\begin{document}

\date{March 11, 2013}
\title{\bf Pointwise Minimum Norm Control Laws for Hybrid Systems}
\author{Ricardo G. Sanfelice\thanks{R. G. Sanfelice is with
the Department of Aerospace and Mechanical Engineering, University of Arizona,
1130 N. Mountain Ave, AZ 85721.
      Email: {\tt\small sricardo@u.arizona.edu}.
      Research partially supported by NSF CAREER Grant no. ECS-1150306 and by AFOSR YIP Grant no. FA9550-12-1-0366.
}
}

\maketitle

%\vspace{-0.3in}
\begin{abstract}
Pointwise minimum norm control laws for hybrid dynamical systems are proposed.
Hybrid systems are given by differential equations
capturing the continuous dynamics or {\em flows}, and by
difference equations capturing the discrete dynamics
or {\em jumps}.
The proposed control laws are defined as
the pointwise minimum norm selection from the set of inputs
guaranteeing a decrease of a control Lyapunov function. 
The cases of individual and common inputs during flows and jumps,
as well as when inputs enter through one of the system dynamics, are considered.
Examples illustrate the results.
\end{abstract}

\section{Introduction}

The construction of asymptotically stabilizing control 
laws from control Lyapunov functions (CLFs) has enabled the
systematic design of feedback laws for nonlinear systems.
Building from earlier results in \cite{Artstein83},
which revealed a key link between the 
availability of a control Lyapunov function and stabilizability (with relaxed controls),
the construction of control laws from Lyapunov inequalities
was rendered as a powerful control design methodology (see also, e.g., 
\cite{Clarke00,Sontag.Sussman.96}, for the connections between CLFs and asymptotic controllability to
the origin).
More importantly, design techniques that
go beyond the possibility of determining the control law 
from the expression of the Lyapunov inequalities 
were proposed and employed in several applications.
The control law introduced in \cite{SontagSYSCON89},
known as Sontag's universal formula,
provides a generic controller construction for nonlinear systems in affine form that (modulo some
extra properties at the origin) only requires the existence of a CLF. 
(Recent extensions to polynomial systems appeared in \cite{Moulay.Perruquetti.05.TAC}).
The constructions introduced in \cite{FreemanKokotovic96}
have the extra property that their pointwise norm is minimum (for a given CLF).
More notably, 
as shown in \cite{FreemanKokotovic96} by making a link between CLFs and the solution to a differential game,
under additional properties, 
pointwise minimum norm control laws guarantee robustness of the closed-loop system.

%and domination redesign \cite{Sepulchre.ea.97,Krstic.Deng.98};
%
%In \cite{Clarke97}, through the construction
%of a nonsmooth control Lyapunov function,
%it is shown that every
%continuous-time system
%that is asymptotically controllable to the origin can be globally stabilized
%by a (discontinuous) feedback law.
%Further results on existence  and equivalences between 
%nonsmooth control Lyapunov functions and  asymptotic controllability
%appeared in \cite{Clarke00,Sontag.Sussman.96,Rifford.01.SIAM,Rifford.01.ESAIM,Kellett.Teel.04.SIAM,Kellett.Teel.04.SCL}.

In this paper, 
pointwise  minimum norm control laws for hybrid dynamical systems are proposed.
Hybrid dynamical systems are given by differential equations
capturing the continuous dynamics or {\em flows}, and by
difference equations capturing the discrete dynamics
or {\em jumps}.
The conditions determining whether flows or jumps should occur
are given in terms of both the state and the inputs.
For this class of hybrid systems, 
control Lyapunov functions are defined by continuously differentiable functions
whose change, both along flows and jumps,
is upper bounded by a negative definite function of the state.
The proposed control law consists of 
a pointwise minimum norm selection from the set of inputs
that guarantees a decrease of the Lyapunov function on each regime. 
We consider the case when the inputs 
acting during flows are different than the inputs acting during jumps,
the case when the inputs are the same, as well as cases when 
inputs affect only the flows or the jumps.
Conditions guaranteeing continuity and globality of the proposed
pointwise minimum norm control laws are also presented.
Our results not only recover the results in \cite{FreemanKokotovic96SIAM}
when specialized to continuous-time systems, but also
provide the discrete-time versions, which do not seem 
available in the literature.

The remainder of the paper is organized as follows.
Section~\ref{sec:CLF}
introduces the framework for hybrid systems, the notion of solution,
and control Lyapunov functions.
Section~\ref{sec:MinNormLaws} presents the results on stabilization by
pointwise minimum norm control laws.
Examples in Section~\ref{sec:Examples} illustrate some of the results.

\medskip

\noindent
{\bf Notation:}
$\reals^{n}$ denotes $n$-dimensional Euclidean space,
$\reals$ denotes the real numbers.
$\realsplus$ denotes the nonnegative real numbers, i.e.,
  $\realsplus=[0,\infty)$.
$\nats$ denotes the natural numbers including $0$, i.e.,
  $\nats=\left\{0,1,\ldots \right\}$.
$\ball$ denotes the closed unit ball in a Euclidean space.
Given a set $K$, $\ol{K}$ denotes its closure.
Given a set $S$, $\partial{S}$ denotes its boundary.
Given  $x\in \reals^n$, $|x|$ denotes the Euclidean vector norm.
Given a set $K\subset \reals^n$ and $x \in \reals^n$,
  $|x|_{K}:= \inf_{y \in K} |x-y|$.
Given $x$ and $y$, $\langle x, y \rangle$ denotes their inner product.
A function $\alpha : \realsplus \to
\realsgeq$ is said to belong to class-${\mathcal K}_{\infty}$
if it is continuous, zero at zero, strictly increasing, and unbounded.
Given a closed set $K \subset \reals^n\times \U_\star$ with  $\star$ being either $c$ or $d$
and $\U_{\star} \subset \reals^{m_\star}$, define 
$\Pi(K):= \defset{x}{
\exists u_\star \in \U_\star
\mbox{ s.t. } (x,u_\star) \in K}$
and
$\Psi(x,K):= \defset{u}{(x,u) \in K}.$
That is, given a set $K$,
$\Pi(K)$ denotes the ``projection'' of $K$ onto $\reals^n$
while, given $x$, $\Psi(x,K)$ denotes the set of values $u$ 
such that $(x,u) \in K$.
Then, for each $x\in\reals^n$, define the set-valued maps
$\Psi_c:\reals^n \rightrightarrows \U_c$,
$\Psi_d:\reals^n \rightrightarrows \U_d$
as
$\Psi_c(x) := \Psi(x,C)$ and 
$\Psi_d(x) := \Psi(x,D)$, respectively.
Given a map $f$, its graph
 is denoted by $\graph(f)$.

%% ASYMPTOTIC CONTROLLABILITY DEFINITION
%\begin{definition}[uniformly globally asymptotically controllability]
%\label{UGAC definition}              
%Let $\A\subset O$ be a compact set and let $\sigma_c,
%\sigma_d:\realsgeq \to \realsgeq$ be nondecreasing functions. The
%hybrid system $\HS$ is uniformly globally asymptotically controllable
%(UGAC) to $\A$ with $\U_c,\U_d \cap \sigma_c, \sigma_d$ controls if
%there exists a function $\beta\in \classKLL$ such that for each
%$x(0,0)\in O$ there exist a solution $(x,u_c,u_d)$ to $\HS$ with
%absolutely continuous signals $u_c:\dom u_c \to \U_C, u_d:\dom u_d \to
%\U_d$ satisfying
%\begin{eqnarray}
%|x(t,j)|_{\A} &\leq& \beta(|x(0,0)|_{\A}, t, j) \\
%|u_c(t,j)|_{\A} &\leq& \sigma_c(|x(t,j)|_{\A}) \\
%|u_d(t,j)|_{\A} &\leq& \sigma_d(|x(t,j)|_{\A})\ ,
%\end{eqnarray}
%for all $(t,j) \in \dom x$.
%\end{definition}         

\section{Preliminaries on Hybrid Systems and Control Lyapunov Functions}
\label{sec:CLF}

In this section, we define
control Lyapunov functions (CLFs) for hybrid systems $\HS$ 
with data $(C,f,D,g)$ and given by
\begin{eqnarray}\label{eqn:HS}
\HS\ \left\{
\begin{array}{llllll}
\dot{x} & = & f(x,u_c)& \qquad & (x,u_c) \in C  \\
x^+ & = & g(x,u_d)& \qquad & (x,u_d) \in D,
\end{array}
\right.
\end{eqnarray}
where 
the set $C \subset \reals^n\times\U_c$ is the {\em flow set},
%the set-valued map $F:\reals^n\times\reals^{m_c} \rightrightarrows \reals^n$ is the {\em flow map},
the map $f:\reals^n\times\reals^{m_c} \to \reals^n$ is the {\em flow map},
the set $D\subset \reals^n\times\U_d$ is the {\em jump set}, and
the map $g:\reals^n \to \reals^n$ is the {\em jump map}.
%the set-valued map $g:\reals^n \rightrightarrows \reals^n$ is the {\em jump map}.
The space for the state is $x \in \reals^n$ and the space for the input $u = (u_c,u_d)$
is $\U = \U_c \times \U_d$, where $\U_c \subset \reals^{m_c}$ and 
$\U_d \subset \reals^{m_d}$.
At times, we will require $\HS$ to satisfy the following mild properties.

\begin{definition}[hybrid basic conditions]
\label{def:HBC}
A hybrid system $\HS$ is said to satisfy the {\em hybrid basic conditions}
if its data $(C,f,D,g)$ is such that
%\footnote{A set-valued map $S:\reals^n\rightrightarrows \reals^m$ is {\it outer
%  semicontinuous} at $x\in\reals^n$ if for each sequence $\{x_{i}\}_{i=1}^{\infty}$ converging to a
%point $x \in \reals^n$ and each sequence $y_{i} \in S(x_{i})$ converging to a
%point $y$, it holds that $y \in S(x)$; see \cite[Definition 5.4]{RockafellarWets98}. 
%Given a set $X \subset \reals^n$, 
%it is {\it outer semicontinuous relative to $X$}
%if the set-valued mapping from $\reals^n$ to $\reals^m$ defined by $S(x)$ 
%for $x \in X$ and $\emptyset$ for $x \not \in X$ is outer semicontinuous at each $x \in X$.
%It is {\it locally bounded} if,
%for each compact set $\K \subset \reals^n$ there exists a compact set $\K'\subset \reals^n$ such that
%$S(\K) := \cup_{x \in \K} S(x) \subset \K'$.}
\begin{list}{}{\itemsep.3cm} 
\item[(A1)] $C$ and $D$ are closed subsets of $\reals^n\times {\cal U}_c$ and $\reals^n\times {\cal U}_d$, respectively;
% and ${\cal U}$ is a closed subset of $\reals^m$.
\item[(A2)] $f:\reals^n\times\reals^{m_c}\to \reals^n$ is continuous;
\item[(A3)] $g:\reals^n\times\reals^{m_d}\to \reals^n$ is continuous.
%\item[(A3)] $G:\reals^n\times\reals^{m_d}\rightrightarrows \reals^n$ is outer semicontinuous relative to $D$  
%  and locally bounded, and for all $(x,u_d) \in D$, $g(x,u_d)$ is nonempty.
\end{list}
\end{definition}

% Solutions prelim
Solutions to hybrid systems $\HS$ are given in terms 
of hybrid arcs and hybrid inputs on hybrid time domains.
Hybrid time domains are 
subsets $\SSS$ of $\realsgeq\times\nats$
that, for each $(T,J)\in \SSS$, 
$\SSS\ \cap\ \left( 
  [0,T]\times\{0,1,...J\} \right)$ 
  can be written as 
  $\cup_{j=0}^{J-1} 
\left([t_j,t_{j+1}],j\right)$
for some finite sequence 
of times $0=t_0\leq t_1 \leq t_2...\leq t_J$.\footnote{This property is to hold at each $(T,J)\in\SSS$, but $\SSS$ can be unbounded.}
A hybrid arc $\phi$ is a function on a hybrid time domain 
that, for each $j\in\nats$, 
$t\mapsto \phi(t,j)$ is absolutely continuous on the interval $\defset{t }{(t,j) \in \dom \phi}$, 
while a hybrid input $u$ is a function on a hybrid time domain 
that, for each $j\in\nats$, $t\mapsto u(t,j)$ is 
Lebesgue measurable and locally essentially bounded on the interval $\defset{t }{(t,j) \in \dom u}$.
Then, a solution to the hybrid system $\HS$ is given by a pair $(\phi,u)$, $u = (u_c,u_d)$,
with $\dom \phi = \dom u (= \dom (\phi,u))$ and satisfying the dynamics of $\HS$,
where $\phi$ is a hybrid arc and $u$ a hybrid input.
A solution pair $(\phi,u)$ to $\HS$ is said to be {\it complete} if $\dom (\phi,u)$ 
is unbounded and {\it maximal} if there does not
exist another pair $(\phi,u)'$ such that $(\phi,u)$ is a truncation of
$(\phi,u)'$ to some proper subset of $\dom (\phi,u)'$. 
For more details about solutions to hybrid systems, see \cite{Sanfelice.10.CDC}.

We introduce the concept of control Lyapunov function for hybrid systems $\HS$;
see \cite{Sanfelice.11.TAC.CLF} for more details and
conditions on $\HS$ guaranteeing its existence.

\begin{definition}[control Lyapunov function]
\label{control Lyapunov function definition}              
Given a compact set $\A \subset \reals^n$
and sets $\U_c \subset \reals^{m_c}, \U_d \subset \reals^{m_d}$,
a continuous function $V:\reals^n\to \reals$,
continuously differentiable on an open set 
containing $\overline{\Pi(C)}$
 is a {\em control Lyapunov function
with $\U$ controls for $\HS$} if
there exist $\alpha_1, \alpha_2\in \classKinfty$
and a positive definite function $\alpha_3$
such that\footnote{Following \cite[Definition 4.1]{FreemanKokotovic96SIAM}, 
\eqref{eqn:CLFFlow} 
can be replaced by
$\inf_{u_c  \in \Psi_c(x)}\ 
\langle \nabla V(x), f(x,u_c) \rangle  <0$
for all $x \in \Pi(C) \setminus \A$,
since, then,
\cite[Proposition 4.3]{FreemanKokotovic96SIAM} guarantees 
the existence of a continuous positive definite function $\alpha_3$ satisfying \eqref{eqn:CLFFlow}
(similarly for 
\eqref{eqn:CLFJump}).}
\begin{eqnarray}
\label{eqn:CLFBounds}
& & \alpha_1(|x|_\A)\ \ \leq\ \ V(x)\ \  \leq\ \ \alpha_2(|x|_\A)
\qquad \qquad \forall x \in \Pi(C)\cup \Pi(D) \cup g(D),
\\
 \label{eqn:CLFFlow}
& & \hspace{-0.2in}\inf_{u_c  \in \Psi_c(x)}\  
%\sup_{\xi \in f(x,u_c)} \langle \nabla V(x), \xi \rangle
\langle \nabla V(x), f(x,u_c) \rangle
  \leq  - \alpha_3(|x|_{\A})
\qquad\qquad\quad \forall x \in \Pi(C),
\\
\label{eqn:CLFJump}
& & \inf_{u_d  \in \Psi_d(x)} 
%\ \sup_{\xi \in g(x,u_d)} 
V(g(x,u_d))  -  V(x)  \leq  - \alpha_3(|x|_{\A})
\qquad\quad \forall x \in \Pi(D). 
\end{eqnarray}
\end{definition}                

\section{Minimum Norm State-Feedback Laws for Hybrid Systems}
\label{sec:MinNormLaws}

Given a hybrid system $\HS$
satisfying the hybrid basic conditions, 
a compact set $\A$, and a control Lyapunov function $V$ 
satisfying Definition~\ref{control Lyapunov function definition}, 
define, for each $r \in \realsgeq$, the set
$$
\Ir := \defset{x \in \reals^n}{V(x) \geq r}.
$$
Moreover, for each $(x,u_c)\in \reals^n\times\reals^{m_c}$ and $r \in \realsgeq$,
define the function
\begin{eqnarray*}
\Gamma_c(x,u_c,r) & : = & 
\left\{
\begin{array}{ll}\displaystyle
\langle \nabla V(x), f(x,u_c) \rangle  + {\alpha}_3(|x|_\A) & 
 \mbox{ if }  (x,u_c) \in C \cap (\Ir\times \reals^{m_c}),\\
-\infty & \mbox{ otherwise }
\end{array}
\right.
\end{eqnarray*}
and, for each $(x,u_d)\in \reals^n\times\reals^{m_d}$ and $r \in \realsgeq$,
the function
\begin{eqnarray*}
\Gamma_d(x,u_d,r) & : = & 
\left\{
\begin{array}{ll}\displaystyle
%\max_{\xi \in g(x,u_d)}
V(g(x,u_d))  -
 V(x) + {\alpha}_3(|x|_\A) &  \mbox{ if } 
    (x,u_d)\in  D \cap (\Ir\times \reals^{m_d}),\\
-\infty & \mbox{ otherwise. }
\end{array}
\right.
\end{eqnarray*}
Then, 
evaluate the functions 
$\Gamma_c$
and 
$\Gamma_d$
at points 
$(x,u_c,r)$
and
$(x,u_d,r)$
where
$r = V(x)$
to
define the functions
\begin{equation}\label{eqn:Upsilons}
\begin{array}{c}
(x,u_c) \mapsto 
\Upsilon_c(x,u_c) 
 := \Gamma_c(x,u_c,V(x)) ,
(x,u_d) \mapsto
 \Upsilon_d(x,u_d)
 :=  \Gamma_d(x,u_d,V(x))
 \end{array}
\end{equation}
and the set-valued maps
\begin{eqnarray}\label{eqn:calTcAndd}
\begin{array}{c}
{\cal T}_c(x) \! : =  \! \Psi_c(x) \cap
\defset{u_c \in \U_c}{ \Upsilon_c(x,u_c) \leq 0},
{\cal T}_d(x)\! : =  \! \Psi_d(x) \cap 
\defset{u_d \in \U_d}{\Upsilon_d(x,u_d) \leq 0}.
\end{array}
\end{eqnarray}
Furthermore, define \begin{equation}\label{eqn:CsetAnyR}
R_c := \Pi(C) \cap \defset{x\in \reals^n}{V(x) > 0}
\end{equation}
and 
\begin{equation}\label{eqn:DsetAnyR}
R_d := \Pi(D) \cap  \defset{x\in \reals^n}{V(x) > 0}.
\end{equation}
When, for each $x$, 
the functions
$u_c \mapsto \Upsilon_c(x,u_c)$
and
$u_d \mapsto \Upsilon_d(x,u_c)$ are convex,
and
the set-valued maps
$\Psi_c$ and $\Psi_d$
have nonempty closed convex values
%and
%${\cal T}_c$ and ${\cal T}_d$
%are lower semicontinuous with nonempty convex values
on $R_c$ and $R_d$, respectively,
we have that 
${\cal T}_c(x)$
and 
${\cal T}_d(x)$
have nonempty convex closed values on
\eqref{eqn:CsetAnyR}
and 
on
\eqref{eqn:DsetAnyR},
respectively (this follows from \cite[Proposition 4.4]{FreemanKokotovic96SIAM}).
Then, 
${\cal T}_c$
and 
${\cal T}_d$
have unique elements of minimum norm on $R_c$ and $R_d$, respectively,
and 
their minimal selections 
\begin{eqnarray*}
\rho_c: R_c \to \U_c, \qquad
\rho_d: R_d \to \U_d
\end{eqnarray*}
are given by
\begin{eqnarray}\label{eqn:mc}
\rho_c(x) := \arg \min \defset{|u_c|}{ u_c \in {{\cal T}}_c(x)},
\\
\label{eqn:md}
\rho_d(x) := \arg \min \defset{|u_d|}{ u_d \in {{\cal T}}_d(x)}.
\end{eqnarray}
Moreover, these selections are continuous under further properties of $\Psi_c$ and $\Psi_d$.

The hybrid system $\HS$ under the effect of the control pair
$(\rho_c,\rho_d)$ 
in \eqref{eqn:mc}, \eqref{eqn:md}
is given by
\begin{eqnarray}\label{eqn:HScl}
\widetilde{\HS}\ \left\{
\begin{array}{llllll}
\dot{x} & = & \widetilde{f}(x):= f(x,\rho_c(x))& \qquad & x \in \widetilde{C}  \\
x^+ & = & \widetilde{g}(x):=  g(x,\rho_d(x))& \qquad & x \in \widetilde{D}
\end{array}
\right.
\end{eqnarray}
with
$\widetilde{C}  :=  \defset{x \in \reals^n}{(x,\rho_c(x)) \in C}$ and
$\widetilde{D}  :=  \defset{x \in \reals^n}{(x,\rho_d(x)) \in D}$.
%\cite[Proposition 2.19]{FreemanKokotovic96}.
%Using these constructions, we presents 
The above arguments and constructions enable the stabilization results in the following sections.

\begin{remark}
\label{rmk:NotCommonAlpha}
%\startmodif
When
bounds \eqref{eqn:CLFFlow} and \eqref{eqn:CLFJump}
hold for functions $\alpha_{3,c}$ and $\alpha_{3,d}$,
then a common function $\alpha_3$ is given by 
$\alpha_3(s) = \min\{\alpha_{3,c}(s),\alpha_{3,d}(s)\}$ for all $s\geq0$.
In such a case,
%while the current construction of the set-valued maps
%${\cal T}_c$
%and
%${\cal T}_d$
%uses a common function $\alpha_3$,
%which is obtained from the definition of CLF in Definition~\ref{control Lyapunov function definition}.
the expressions of the pointwise minimum norm control laws
\eqref{eqn:mc}
and
\eqref{eqn:md}
could also be given in terms of 
$\alpha_{3,c}$ and $\alpha_{3,d}$
by defining ${\cal T}_c$ and ${\cal T}_d$
in terms of $\alpha_{3,c}$ and $\alpha_{3,d}$, respectively.
\end{remark}

\subsection{Practical stabilization using min-norm hybrid control}

Proposition~\ref{prop:MinNormPracticalStabilization} below
establishes that the pointwise minimum norm controller in \eqref{eqn:mc}-\eqref{eqn:md}
asymptotically stabilizes the compact set\footnote{A compact set $\A$ is said to be asymptotically stable
for a closed-loop system (e.g., $\widetilde{\HS}$ in \eqref{eqn:HScl}) if: $\bullet$ for each $\eps > 0$ there exists
$\delta >0$ such that 
each maximal solution $\phi$ starting from
$\A + \delta \ball$
satisfies $\phi(t,j) \in \A+\eps\ball$
for each $(t,j)\in\dom \phi$, and $\bullet$
each maximal solution is bounded and the 
complete ones satisfy $\lim_{t+j\to\infty}|\phi(t,j)|_{\A}=0$.
} 
\begin{equation}\label{eqn:SetAr}
\A_r := \defset{x \in \reals^n}{V(x) \leq r}
\end{equation}
for the  hybrid system restricted to $\Ir$.
More precisely,
given $r > 0$,
we restrict the flow and jump sets
of the hybrid system $\HS$ by the set $\Ir$, 
which leads to
\begin{eqnarray*}
\HS_{\cal I}\ \left\{
\begin{array}{llllll}
\dot{x} & = & f(x,u_c)& \ & (x,u_c) \in C  \cap (\Ir\times \reals^{m_c}) \\
x^+ & = & g(x,u_d)& \ & (x,u_d) \in D  \cap (\Ir\times \reals^{m_d}).
\end{array}
\right.
\end{eqnarray*}

\begin{proposition}
\label{prop:MinNormPracticalStabilization}
Given a compact set $\A \subset \reals^n$ and 
a hybrid system $\HS = (C,f,D,g)$ satisfying the hybrid basic conditions,
suppose there exists a control Lyapunov function $V$
with ${\cal U}$ controls for $\HS$.
Furthermore, suppose the following conditions hold:
\begin{list}{}{\itemsep.3cm} 
\item[(M1)]
The set-valued maps $\Psi_c$ and $\Psi_d$
are lower semicontinuous\footnote{A set-valued map $S:\reals^n\rightrightarrows\reals^m$
is lower semicontinuous if
for each $x \in \reals^n$ 
one has that 
$
\liminf_{x_i \to x} S(x_i) \supset S(x)
$,
where $\liminf_{x_i \to x} S(x_i) = \defset{z}{\forall x_i \to x, \exists z_i \to z \mbox{ s.t. } z_i \in S(x_i)}$
is the {\em inner limit} of $S$ (see \cite[Chapter 5.B]{RockafellarWets98}).} 
with convex values.
%\item[M2)]
%The functions
%$
%\Gamma_c
%$
%and 
%$
%\Gamma_d
%$
%are upper semicontinuous.
\item[(M2)]
For every $r>0$ and every $x \in \Pi(C) \cap \Ir$,
the function
$u_c \mapsto \Gamma_c(x,u_c,r)$ is convex on $\Psi_c(x)$
and, for every $r>0$ and every $x \in \Pi(D) \cap \Ir$,
the function
$u_d \mapsto \Gamma_c(x,u_d,r)$ is convex on $\Psi_d(x)$.
\end{list}
Then, for every $r>0$, the state-feedback law pair
\begin{eqnarray*}
\rho_c: R_c \cap \Ir \to \U_c, \qquad
\rho_d: R_d \cap \Ir \to \U_d
\end{eqnarray*}
defined as
\begin{eqnarray}\label{eqn:mc-practical}
\rho_c(x) &:=& \arg \min \defset{|u_c|}{ u_c \in {{\cal T}}_c(x)}
 \qquad \forall x \in R_c \cap \Ir,
\\
\label{eqn:md-practical}
\rho_d(x) &:=& \arg \min \defset{|u_d|}{ u_d \in {{\cal T}}_d(x)}  \qquad \forall x \in R_d \cap \Ir
\end{eqnarray}
%is continuous and 
renders the compact set 
$\A_r$ asymptotically stable for $\HS_{\cal I}$.
Furthermore, if 
the set-valued maps $\Psi_c$ and $\Psi_d$
have closed graph then $\rho_c$ and $\rho_d$ are continuous.
\end{proposition}
\begin{proof}
Given the CLF $V$ for $\HS$,
by Definition~\ref{control Lyapunov function definition},
we have \eqref{eqn:CLFBounds} and, for every $r > 0$,
\eqn{
 \label{eqn:CLFFlow-HSr}
\inf_{u_c  \in \Psi_c(x)}
\langle \nabla V(x), f(x,u_c) \rangle  & \leq &  - \alpha_3(|x|_{\A})
\qquad\qquad\quad \forall x \in \Pi(C) \cap \Ir,
\\
\label{eqn:CLFJump-HSr}
\inf_{u_d  \in \Psi_d(x)}  
%\max_{\xi \in g(x,u_d)} 
V(g(x,u_d))  -  V(x) & \leq & - \alpha_3(|x|_{\A})
\qquad\qquad\quad \forall x \in \Pi(D) \cap \Ir,
}
from where $\Gamma_c$ and $\Gamma_d$
are defined.
Using the continuity properties of $f$ and
$g$ obtained from (A2) and (A3) of the hybrid basic conditions,
%Lemma~\ref{lemma lower semicontinuity of sets},
and 
continuous differentiability of $V$,
 it follows that, for every $r \geq 0$,
$\Gamma_c$ and $\Gamma_d$ are continuous on
$C \cap (\Ir\times \reals^{m_c})$ and on $D \cap (\Ir\times \reals^{m_d})$, respectively.
Since $C$ and $D$ are closed by (A1) of the hybrid basic conditions and $V$ is continuous,
the sets
$C \cap (\Ir \times \reals^{m_c})$ and $D \cap (\Ir \times \reals^{m_d})$ are closed
for each $r$.
By the closedness property of
$C$ and $D$ along with assumption (M1),
the set-valued maps $\Psi_c$ and $\Psi_d$ have 
nonempty closed convex values on $R_c$ and $R_d$, respectively.
Using 
(M2),
%continuity of $V$,
%and the hybrid basic conditions,
the functions
$u_c \mapsto \Upsilon_c(x,u_c)$
and
$u_d \mapsto \Upsilon_d(x,u_c)$ 
defined in \eqref{eqn:Upsilons}
are convex on $R_c$ and $R_d$, respectively.
%and
%${\cal T}_c$ and ${\cal T}_d$
%are lower semicontinuous with nonempty convex values
Then, \cite[Proposition 4.4]{FreemanKokotovic96SIAM}
implies that ${\cal T}_c$ and ${\cal T}_d$
are lower semicontinuous with nonempty closed convex values
on $R_c$ and $R_d$, respectively.
%% from TAC CLF submission
%by lower semicontinuity of the
%set-valued maps $\Psi_c$ and $\Psi_d$ in R1)
%and by upper semicontinuity of the functions $\Gamma_c$ and $\Gamma_d$ in M2),
%we have that $\widetilde{S}_c$ and $\widetilde{S}_d$ are lower semicontinuous.
%This property follows by 
%\cite[Corollary 2.13]{FreemanKokotovic96}
%%Lemma~\ref{lemma lower semicontinuity of sets}
%with $z = (x,r)$, $z' = u_\star$, $W(z) = \Psi_\star(x)$, and $w = \Gamma_\star$.
%By  \eqref{eqn:CLFFlow}-\eqref{eqn:CLFJump} 
%and the construction of $\Gamma_\star$, we have that, for each $r >0$,
%$\widetilde{S}_c$ and $\widetilde{S}_d$
%are nonempty
%on $\Pi(C) \cap \Ir$ and 
%on $\Pi(D) \cap \Ir$, respectively.
%By the 
%convexity property of the functions $\Gamma_c$ and $\Gamma_d$ in M3)
%and of the values of the set-valued maps $\Psi_c$ and $\Psi_d$ in R1),
%we have that, for each $r > 0$, 
%$\widetilde{S}_c$ and $\widetilde{S}_d$
%are  convex valued 
%on $\Pi(C) \cap \Ir$ and 
%on $\Pi(D) \cap \Ir$, respectively.
Moreover, 
${\cal T}_c$ and ${\cal T}_d$
have unique elements of minimum norm,
and their minimal selections 
\begin{eqnarray}
& & \rho_c: R_c \to \U_c\\
& & \rho_d: R_d \to \U_d
\end{eqnarray}
on 
\eqref{eqn:CsetAnyR}
and 
on
\eqref{eqn:DsetAnyR}
(respectively)
are,
by definition, 
given by
\eqref{eqn:mc-practical}
and 
\eqref{eqn:md-practical} (respectively).
Then, 
from \eqref{eqn:mc-practical} and \eqref{eqn:md-practical}, 
we have
\begin{eqnarray*}
& & \rho_c(x) \in \Psi_c(x), \quad \Upsilon_c(x,\rho_c(x)) \leq 0\quad \forall x \in R_c \cap \Ir \\
& & \rho_d(x) \in \Psi_d(x), \quad \Upsilon_d(x,\rho_d(x)) \leq 0\quad \forall x \in R_d \cap \Ir.
\end{eqnarray*}
Using the definitions of 
$\Psi_c, \Psi_d$
and 
$\Upsilon_c, \Upsilon_d$, we have
\eqn{
 \label{eqn:CLFFlow-HSr-mc}
\langle \nabla V(x), f(x,\rho_c(x)) \rangle  & \leq &  - \alpha_3(|x|_{\A})
\qquad\qquad\quad \forall x \in \Pi(C) \cap \Ir,
\\
\label{eqn:CLFJump-HSr-md}
%\max_{\xi \in g(x,\rho_d(x))} 
V(g(x,\rho_d(x)))  -  V(x) & \leq & - \alpha_3(|x|_{\A})
\qquad\qquad\quad \forall x \in \Pi(D) \cap \Ir.
}
Then, for every $r>0$, we have a state-feedback pair
$(\rho_c,\rho_d)$
that
renders the compact set $\A_r$ asymptotically
stable for $\HS_{\cal I}$.  
This property
follows
% from the 
%fact that
%the hybrid system $\HS_{\cal I}$
%satisfies the hybrid basic conditions
%and 
from an application of the
Lyapunov
stability result in 
\cite[Theorem 3.18]{Goebel.ea.11}.
%\cite[Corollary 7.7]{SanfeliceGoebelTeel05}.

If $\Psi_c$ and $\Psi_d$ have closed
graph, then
we have that
the graph of ${\cal T}_c$ and ${\cal T}_d$
are closed since,
for $\star = c,d$,
$$\graph({\cal T}_\star)\! =\! \graph(\Psi_\star(x)) \cap \graph(\defset{u_\star \in \U_\star}{ \Upsilon_\star(x,u_\star) \leq 0\!\!}),$$
where the first graph is closed by assumption while the second one is closed by the closedness and continuity properties
of $\U_\star$ and $\Upsilon_\star$, respectively.
%see \cite[Proposition 7.1]{FreemanKokotovic96SIAM})
Then, using \cite[Proposition 2.19]{FreemanKokotovic96},
the minimal selections 
\begin{eqnarray*}
\rho_c: R_c \to \U_c, \qquad
\rho_d: R_d \to \U_d
\end{eqnarray*}
on 
\eqref{eqn:CsetAnyR}
and 
on
\eqref{eqn:DsetAnyR},
which are given by 
\eqref{eqn:mc-practical}
and 
\eqref{eqn:md-practical}, respectively,
are continuous.\footnote{Note that by the hybrid basic conditions of $\HS$, continuity
of $\rho_c$ and $\rho_d$, and closedness of $\Ir$,
the hybrid system $\HS_{\cal I}$ 
with the control laws
\eqref{eqn:mc-practical}
and 
\eqref{eqn:md-practical}
applied to it 
satisfies the hybrid basic conditions.}
\end{proof}

\begin{remark}
The state-feedback law 
\eqref{eqn:mc-practical}-\eqref{eqn:md-practical}
asymptotically stabilizes $\A_r$ for $\HS_{\cal I}$
(but not necessarily for $\HS$ as without 
an appropriate extension of these laws to $\Pi(C)$ and $\Pi(D)$, respectively,
there could exist solutions to the closed-loop system
that jump out of $\A_r$).
This point motivates the
following result on stabilization by 
a control law that has pointwise minimum norm at points in ${\cal I}(r)$, but not everywhere,
and 
the
global stabilization 
result in the next section.
Finally, note that the assumptions placed on $\HS$, such as the 
existence of a CLF,
can be relaxed by imposing them on $\HS_{\cal I}$
instead.
\end{remark}

\begin{theorem}
\label{thm:MinNormStabilizationOnIr}
Under the conditions of Proposition~\ref{prop:MinNormPracticalStabilization},
for every $r >0$ there exists a 
state-feedback law pair
\begin{eqnarray*}
\rho'_c: R_c \to \U_c, \qquad
\rho'_d: R_d \to \U_d
\end{eqnarray*}
defined on $R_c \cap \Ir$ and $R_d \cap \Ir$ as
\begin{eqnarray}\label{eqn:mc-practical2}
\rho'_c(x) &:=& \arg \min \defset{|u_c|}{ u_c \in {{\cal T}}_c(x)} 
\qquad \forall x \in R_c \cap \Ir,
\\
\label{eqn:md-practical2}
\rho'_d(x) &:=& \arg \min \defset{|u_d|}{ u_d \in {{\cal T}}_d(x)}
 \qquad \forall x \in R_d \cap \Ir
\end{eqnarray}
%is continuous and 
respectively,
that renders the compact set 
$\A_r$ asymptotically stable for $\HS$.
Furthermore, if 
the set-valued maps $\Psi_c$ and $\Psi_d$
have closed graph then $\rho'_c$ and $\rho'_d$ are continuous on 
$R_c \cap \Ir$ and $R_d \cap \Ir$, respectively.
\end{theorem}

The result follows using Proposition~\ref{prop:MinNormPracticalStabilization}
and the fact that,
from the definition of CLF in Definition~\ref{control Lyapunov function definition},
since the right-hand side of \eqref{eqn:CLFFlow} is negative definite with respect to $\A$
(respectively, \eqref{eqn:CLFJump})
the state-feedback $\rho_c$ (respectively, $\rho_d$) in 
\eqref{eqn:mc} (respectively, \eqref{eqn:md})
can be extended -- not necessarily as a pointwise minimum norm law -- to every point in $\Pi(C) \cap \A_r$ 
(respectively, $\Pi(D) \cap \A_r$) and guarantee that $V$ is nonincreasing.
The asymptotic stability of $\A_r$ for $\HS$ then follows from an application of
\cite[Theorem 3.18]{Goebel.ea.11}.  
Finally, as the definition of 
${{\cal T}}_c$
and 
${{\cal T}}_d$
suggest, 
the 
norm-minimality of $\rho_c$ and $\rho_d$ are functions
of $V$ and $\alpha_3$, and different such choices would give
different pointwise minimum norm control laws.

\subsection{Global stabilization using min-norm hybrid control}

The result in the previous section guarantees a practical stability property
through the use of a pointwise minimum norm state-feedback control law.
Now, we consider the global stabilization of a compact
set via continuous state-feedback laws $(\rho_c,\rho_d)$ with pointwise minimum norm.
% which leads to the closed-loop system 
%\begin{eqnarray}\label{eqn:HScl}
%\widetilde{\HS}\ \left\{
%\begin{array}{llllll}
%\dot{x} & = &  f(x,\rho_c(x))& \qquad & x \in \widetilde{C}  \\
%x^+ & \in & g(x,\rho_d(x))& \qquad & x \in \widetilde{D}.
%\end{array}
%\right.
%\end{eqnarray}
For such a purpose, extra conditions 
are required to hold nearby the compact set.
For continuous-time systems, such conditions correspond
to the so-called 
{\em continuous control property} and
 {\em small control property} \cite{SontagSYSCON89,FreemanKokotovic96,Krstic.Deng.98}.
To that end, given a compact set $\A$ and a control Lyapunov function $V$ satisfying Definition~\ref{control Lyapunov function definition},
for each $x \in \reals^n$, 
define
\begin{eqnarray}\label{eqn:Tprimec}
{\cal T}'_c(x) &:=& 
\Psi_c(x) \cap S'_c(x,V(x)), \\ \label{eqn:Tprimed}
{\cal T}'_d(x) &:=& 
\Psi_d(x) \cap S'_d(x,V(x)),
\end{eqnarray}
where, for each $x \in \reals^n$ and each $r \geq 0$,
\begin{eqnarray}\label{eqn:SGlobal}
\begin{array}{l}
{S}'_c(x,r) :=
\left\{
\begin{array}{lll}
S^\circ_c(x,r) & \mbox{ if }  r > 0,\\
\rho_{c,0}(x) & \mbox{ if } r=0,
\end{array}
\right.\ \
{S}'_d(x,r)  :=
\left\{
\begin{array}{lll}
S^\circ_d(x,r) & \mbox{ if }  r > 0,\\
\rho_{d,0}(x) & \mbox{ if } r=0,
\end{array}
\right.
\end{array}
\end{eqnarray}
\begin{eqnarray*}
S^\circ_c(x,r) & = & 
\left\{
\begin{array}{ll}
\defset{u_c \in \U_c}{ \Gamma_c(x,u_c,r) \leq 0} &
 \mbox{ if } x \in \Pi(C) \cap \Ir,\\
\reals^{m_c} & \mbox{ otherwise},
\end{array}
\right. \\
S^\circ_d(x,r) & = & 
\left\{
\begin{array}{ll}
\defset{u_d \in \U_d}{ \Gamma_d(x,u_d,r) \leq 0} &
 \mbox{ if } x \in \Pi(D) \cap \Ir,\\
\reals^{m_d} & \mbox{ otherwise},
\end{array}
\right.
\end{eqnarray*}
and
the feedback law pair
$$ \rho_{c,0}:\reals^n \to  \U_c,\qquad \rho_{d,0}:\reals^n \to \U_d$$
%is such that
%$(\rho_{c,0},\rho_{d,0})(\A) = 0$
%%$\kappa_{c,0}:\reals^n \to \U_c$ and $\kappa_{d,0}:\reals^n \to \U_d$
%and 
induces (strong) forward invariance of $\A$, that is,
\begin{list}{}{\itemsep.3cm} 
\item[(M3)]
Every maximal solution 
$t \mapsto \phi(t,0)$
to 
$
\dot{x} = f(x,\rho_{c,0}(x))$, $x \in \Pi(C)\cap \A
$
satisfies $|\phi(t,0)|_\A = 0$ for all $(t,0) \in \dom \phi$;
\item[(M4)]
Every maximal solution 
$j \mapsto \phi(0,j)$ 
to 
$
x^+ = g(x,\rho_{d,0}(x))$, $x \in \Pi(D)\cap \A
$
satisfies $|\phi(0,j)|_\A = 0$ for all $(0,j) \in \dom \phi$.
\end{list}
Under the conditions in Proposition~\ref{prop:MinNormPracticalStabilization},
the maps in \eqref{eqn:SGlobal}
are lower semicontinuous for every $r > 0$.  
To be able to make continuous selections at $\A$,
these maps are further required to be lower semicontinuous
for $r=0$.
These conditions resemble those already reported in \cite{FreemanKokotovic96}
for continuous-time systems.

\begin{theorem}
\label{thm:MinNormStabilization}
Given a compact set $\A \subset \reals^n$ and 
a hybrid system $\HS = (C,f,D,g)$ satisfying the hybrid basic conditions,
suppose there exists a control Lyapunov function $V$
with ${\cal U}$ controls for $\HS$.
Moreover, suppose
that conditions (M1)-(M2) of Proposition~\ref{prop:MinNormPracticalStabilization} hold.
If the
feedback law pair 
$(\rho_{c,0}:\reals^n \to  \U_c$, $\rho_{d,0}:\reals^n \to \U_d)$
is such that
%$(\rho_{c,0},\rho_{d,0})(\A) = 0$
%and
conditions (M3) and (M4) hold, and
\begin{list}{}{\itemsep.3cm} 
\item[(M5)]
The set-valued map ${\cal T}'_c$ in \eqref{eqn:Tprimec}
is lower semicontinuous at
each $x\in \Pi(C)\cap \Irzero$,
\item[(M6)]
The set-valued map ${\cal T}'_d$  in \eqref{eqn:Tprimed}
is lower semicontinuous at
each $x\in \Pi(D)\cap \Irzero$
\end{list}
hold,
then the state-feedback law pair
\begin{eqnarray*}
\rho_c: \Pi(C) \to \U_c, \qquad 
\rho_d: \Pi(D) \to \U_d
\end{eqnarray*}
defined as
\begin{eqnarray}\label{eqn:mc-global}
\rho_c(x) := \arg \min \defset{|u_c|}{ u_c \in {{\cal T}}'_c(x)} \ \ \forall x \in \Pi(C)
\\
\label{eqn:md-global}
\rho_d(x) := \arg \min \defset{|u_d|}{ u_d \in {{\cal T}}'_d(x)} \ \ \forall x \in \Pi(D)
\end{eqnarray}
renders the compact set $\A$ globally asymptotically stable for $\HS$.
Furthermore,
if the set-valued maps $\Psi_c$ and $\Psi_d$
have closed graph 
and 
$(\rho_{c,0},\rho_{d,0})(\A) = 0$
then $\rho_c$ and $\rho_d$ are continuous.
\end{theorem}
\begin{proof}
The proof follows the ideas of the proof of \cite[Proposition 7.1]{FreemanKokotovic96SIAM}.
Proceeding as in the proof of Proposition~\ref{prop:MinNormPracticalStabilization},
using (M5) and (M6), we have that
${\cal T}'_c$ and ${\cal T}'_d$ 
are lower semicontinuous with nonempty closed values
on $\Pi(C)$ and $\Pi(D)$, respectively.
Then, 
${\cal T}'_c$ and ${\cal T}'_d$
have unique elements of minimum norm,
and their minimal selections 
\begin{eqnarray}
& & \rho_c: \Pi(C) \to \U_c\\
& & \rho_d: \Pi(D) \to \U_d
\end{eqnarray}
on 
$\Pi(C)$
and 
$\Pi(D)$
(respectively)
are given by
\eqref{eqn:mc-global}
and 
\eqref{eqn:md-global} (respectively).
Then, 
from \eqref{eqn:mc-global} and \eqref{eqn:md-global}, 
we have
\begin{eqnarray*}
& & \rho_c(x) \in \Psi_c(x), \ \ \Gamma_c(x,\rho_c(x),V(x)) \leq 0\ \ \ \forall x \in \Pi(C) \\
& & \rho_d(x) \in \Psi_d(x), \ \ \Gamma_d(x,\rho_d(x),V(x)) \leq 0\ \ \ \forall x \in \Pi(D).
\end{eqnarray*}
Using the definitions of 
$\Psi_c, \Psi_d$
and 
$\Gamma_c, \Gamma_d$, we have
\eqn{\nonumber
% \label{eqn:CLFFlow-HSr-mc-global}
\langle \nabla V(x), f(x,\rho_c(x)) \rangle  & \leq &  - \alpha_3(|x|_{\A})
\qquad \forall x \in \Pi(C),
\\
\nonumber
%\label{eqn:CLFJump-HSr-md-global}
%\max_{\xi \in g(x,\rho_d(x))} 
V(g(x,\rho_d(x)))  -  V(x) & \leq & - \alpha_3(|x|_{\A})
\qquad \forall x \in \Pi(D).
}
Then, 
the set $\A$ is 
globally asymptotically stable for the closed-loop system $\widetilde{\HS}$
by an application of the Lyapunov stability theorem for hybrid systems
\cite[Theorem 3.18]{Goebel.ea.11}.
%% OLD way to prove asymptotic stability
%Proposition~\ref{prop:MinNormPracticalStabilization} establishes
%that, for every $r >0$, solutions starting from points $x  \not \in \A$, 
%$V(x) = r$, converge to $\A_r$.
%Using M4) and M5), 
%solutions to $\widetilde{\HS}$ cannot leave $\A$ from points $x \in \A$.
%Then, $\A$ is 
%uniformly attractive and
%forward invariant for the closed-loop system.

When 
the set-valued maps $\Psi_c$ and $\Psi_d$
have closed graph, from
Proposition~\ref{prop:MinNormPracticalStabilization} we have that 
$\rho_c$ and $\rho_d$ are continuous 
on $\Pi(C)\setminus \A$ and
on $\Pi(D)\setminus \A$, respectively.
Moreover, 
if
$(\rho_{c,0},\rho_{d,0})(\A) = 0$,
 \cite[Theorem~4.5]{Sanfelice.11.TAC.CLF}
implies that there exists a continuous feedback pair 
$(\kappa_c,\kappa_d)$ -- not necessarily of pointwise minimum norm --
asymptotically stabilizing the compact set $\A$
and with the property $(\kappa_c,\kappa_d)(\A) = 0$ 
(the pair $(\kappa_c,\kappa_d)$ vanishes on
$\A$ due to the fact that 
the only possible selection for $r=0$ is 
the pair $(\rho_{c,0},\rho_{d,0})$, which vanishes at such points).
Since $\rho_c$ and $\rho_d$ have pointwise minimum norm,
we have
\begin{eqnarray}
& & 0\leq |\rho_c(x)| \leq |\kappa_c(x)| \qquad \forall x \in \Pi(C) \\
& & 0\leq |\rho_d(x)| \leq |\kappa_d(x)| \qquad \forall x \in \Pi(D).
\end{eqnarray}
Then, since $\kappa_c$ and $\kappa_d$ are continuous and vanish at points in $\A$, 
the laws $\rho_c$ and $\rho_d$ are continuous on $\Pi(C)$ and $\Pi(D)$, respectively.
%% OLD way to prove asymptotic stability
%Since, by Lemma~\ref{lemma:HBCwithContinuousFeedback},
%$\widetilde{\HS}$ satisfies the hybrid basic conditions,
%the claim follows from \cite[Proposition 6.1]{GoebelTeel06}.
\end{proof}

\subsection{The case when the inputs affect only flows or only jumps}

% only flows
The results in the previous sections also hold when 
inputs only affect either the flows or jumps, but not both.
In particular, we consider the special case 
when $u_c$ is the only input, in which case
$\HS$ becomes
\begin{eqnarray}\label{eqn:HSflowinput}
\HS_{c}\ \left\{
\begin{array}{lllrlll}
\dot{x} & = & f(x,u_c)& \qquad  (x,u_c) &\!\!\!\! \in C  \\
x^+ & = & g(x)& \qquad  x &\!\!\!\! \in D
\end{array}
\right.
\end{eqnarray}
with
%the set $C \subset \reals^n\times\U_c$ is the {\em flow set},
%the map $f:\reals^n\times\reals^{m_c} \to \reals^n$ is the {\em flow map},
%the set 
$D\subset \reals^n$
% is the {\em jump set}, 
and 
$g:\reals^n \to \reals^n$.
%is the {\em jump map}.
When the only input is $u_d$, $\HS$ becomes
\begin{eqnarray}\label{eqn:HSjumpinput}
\HS_d\ \left\{
\begin{array}{lllrll}
\dot{x} & = & f(x)& \qquad  x &\!\!\!\!\in C  \\
x^+ & = & g(x,u_d)& \qquad  (x,u_d) &\!\!\!\!\in D
\end{array}
\right.
\end{eqnarray}
with, in this case,
$C \subset \reals^n$ and $f:\reals^n\to \reals^n$.
The following results follow by combining the earlier results.

\begin{corollary}
\label{coro:MinNormPracticalStabilizationFlow}
Given a compact set $\A \subset \reals^n$ and 
a hybrid system $\HS_c = (C,f,D,g)$ as in \eqref{eqn:HSflowinput} satisfying the hybrid basic conditions,
suppose there exists a control Lyapunov function $V$
with ${\cal U}$ controls for $\HS_c$.
Furthermore, 
suppose the following conditions hold:
\begin{list}{}{\itemsep.3cm} 
\item[(M1c)]
The set-valued map $\Psi_c$
is lower semicontinuous
with convex values.
\item[(M2c)]
For every $r>0$ and every $x \in \Pi(C) \cap \Ir$,
the function
$u_c \mapsto \Gamma_c(x,u_c,r)$ is convex on $\Psi_c(x)$.
\end{list}
Then, for every $r>0$, there exists a state-feedback law
\begin{eqnarray}
& & \rho'_c: \Pi(C) \to \U_c
\end{eqnarray}
defined on 
$R_c \cap \Ir$
 as in 
\eqref{eqn:mc-practical2}
that renders the compact set 
$\A_r$ asymptotically stable for $\HS_c$.
Moreover, if
the set-valued map $\Psi_c$ has a closed graph then $\rho'_c$ is continuous on $\Pi(C) \cap \Ir$.
Furthermore, if 
the zero feedback law  
$\rho_{c,0}:\reals^n \to \{0\} \subset \U_c$
is such that 
condition (M3) holds and if (M5) holds, then $\rho_c$ in \eqref{eqn:mc-global}
is globally asymptotically stabilizing.  
Furthermore, 
if the set-valued map $\Psi_c$ 
has closed graph then $\rho_c$ is continuous.
\end{corollary}

% only jumps

\begin{corollary}
\label{coro:MinNormPracticalStabilizationJump}
Given a compact set $\A \subset \reals^n$ and 
a hybrid system $\HS_d = (C,f,D,g)$ as in \eqref{eqn:HSjumpinput} satisfying the hybrid basic conditions,
suppose there exists a control Lyapunov function $V$
with ${\cal U}$ controls for $\HS_d$.
Furthermore, suppose the following conditions hold:
\begin{list}{}{\itemsep.3cm} 
\item[(M1d)]
The set-valued map $\Psi_d$
is lower semicontinuous
with convex values.
\item[(M2d)]
For every $r>0$ and every $x \in \Pi(D) \cap \Ir$,
the function
$u_d \mapsto \Gamma_d(x,u_d,r)$ is convex on $\Psi_d(x)$.
\end{list}
Then, for every $r>0$, there exists a state-feedback law
\begin{eqnarray}
& & \rho'_d: \Pi(D) \to \U_d
\end{eqnarray}
defined on 
$R_d \cap \Ir$
 as in \eqref{eqn:md-practical2}
that renders the compact set 
$\A_r$ asymptotically stable for $\HS_d$.
Moreover, if 
the set-valued map $\Psi_d$
has a closed graph then $\rho'_d$ is continuous on $\Pi(D) \cap \Ir$.
Furthermore, if 
the zero feedback law  
$\rho_{d,0}:\reals^n \to \{0\} \subset \U_d$
is such that 
condition (M4) holds and if (M6) holds, then $\rho_d$ in \eqref{eqn:md-global}
is globally asymptotically stabilizing.  
Furthermore, 
if the set-valued map $\Psi_d$ 
has closed graph then $\rho_d$ is continuous.
\end{corollary}

\subsection{The common input case}

When the input for flows and jumps are the same, i.e., 
$u := u_c = u_d$ ($m:= m_c = m_d$),
the hybrid system $\HS$ becomes
\begin{eqnarray}\label{eqn:HScommon}
\HS\ \left\{
\begin{array}{llllll}
\dot{x} & = & f(x,u)& \qquad & (x,u) \in C  \\
x^+ & = & g(x,u)& \qquad & (x,u) \in D
\end{array}
\right.
\end{eqnarray}
and
a common pointwise minimum norm control law 
exists when
\begin{equation}\label{eqn:IntersectionTcprimeTdprimeCondition}
{\cal T}'_c(x) \cap {\cal T}'_d(x) \not = \emptyset
\qquad  \forall x \in \Pi(C) \cap \Pi(D) \cap \Ir
\end{equation}
for each $r$.
% (taking value in the appropriate range).
A result paralleling 
%Proposition~\ref{thm:CLFimpliesStabilizabilityGeneral}
%and Theorem~\ref{thm:CLFimpliesStabilizabilityGeneralGlobal}
Theorem~\ref{thm:MinNormStabilization}
follows
using 
%the set-valued map
$$
{\cal T}'(x) \!:=\!
\left\{
\begin{array}{lll}
{\cal T}'_c(x) &\! \mbox{if } x \in \left(\Pi(C)\setminus \Pi(D)\right) \cap \Ir
\\
{\cal T}'_c(x) \cap {\cal T}'_d(x) &\! \mbox{if } x \in \Pi(C)\cap \Pi(D) \cap \Ir
\\
{\cal T}'_d(x) &\! \mbox{if } x \in \left(\Pi(D)\setminus \Pi(C)\right) \cap \Ir
\\
\reals^m &\! \mbox{otherwise},
\end{array}
\right.
$$
which, when further assuming \eqref{eqn:IntersectionTcprimeTdprimeCondition},
is lower semicontinuous and has nonempty, convex values.
(The set valued map $\cal T$ can be defined similarly.)

\begin{corollary}
Given a compact set $\A \subset \reals^n$ and 
a hybrid system $\HS = (C,f,D,g)$ as in 
\eqref{eqn:HScommon}
satisfying 
the hybrid basic conditions,
suppose there exists a control Lyapunov function $V$
with ${\cal U}$ controls for $\HS$ with input $u = u_c = u_d$ ($m = m_c = m_d$).
Suppose that conditions (M1)-(M2) of Proposition~\ref{prop:MinNormPracticalStabilization}
and condition \eqref{eqn:IntersectionTcprimeTdprimeCondition} hold.
Then, for every $r>0$,
there exists a state-feedback law 
\begin{eqnarray}
& & \rho': \Pi(C) \cup \Pi(D) \to \U
\end{eqnarray}
defined on $(\Pi(C) \cup \Pi(D)) \cap \Ir$ as
\begin{eqnarray*}
\rho'(x) &:=& \arg \min \defset{|u|}{ u \in {{\cal T}}(x)}
  \qquad \forall x \in (\Pi(C) \cup \Pi(D)) \cap \Ir
\end{eqnarray*}
%\eqref{eqn:mc-practical2}
%and
%\eqref{eqn:md-practical2}, 
that renders the compact set 
$\A_r$ asymptotically stable for $\HS$.
Moreover, if
the set-valued maps $\Psi_c$ and $\Psi_d$ have closed graph then $\rho'$ is continuous on $(\Pi(C) \cup \Pi(D)) \cap \Ir$.
Furthermore, 
if the zero feedback law 
$\rho_{0}:\reals^n \to \{0\} \subset \U$
is such that 
\eqref{eqn:IntersectionTcprimeTdprimeCondition} and (M3)-(M6) for $r=0$ 
hold,
then the state-feedback law
\begin{eqnarray}
& & \rho: \Pi(C) \cup \Pi(D) \to \U
\end{eqnarray}
defined as
\begin{equation}
\rho(x) := \arg \min \defset{|u|}{ u \in {{\cal T}}'(x)} \qquad \forall x \in \Pi(C) \cup \Pi(D)
\end{equation}
renders the compact set 
$\A$ globally asymptotically stable for $\HS$.
Furthermore, 
if the set-valued maps $\Psi_c$ and $\Psi_d$
have closed graph then $\rho$ is continuous.
\end{corollary}

\section{Examples}
\label{sec:Examples}

Now, we present examples illustrating some of the
results in the previous sections. Complete details
are presented for the first example.
%More examples and complete details are available in the long
%version of this submission at
%\cite{Sanfelice.13.TR}.

%%%%%%%%%%%%%%%%%%%%%%%%%%%%%%%%%%%%%%%%%%%%%  EXAMPLE 1 %%%%%%%%%%%%%%%%%%%%%%%%%%%%%%%%%%%%%%%%%%%%
\begin{example}[Rotate and dissipate]
\label{ex:2}
Given $v_1, v_2 \in \reals^2$,
let
%\begin{eqnarray}\non
${\cal W}(v_1,v_2) :=
\{\xi \in \reals^2\ : \ 
\xi = r (\lambda v_1 + (1-\lambda) v_2), r \geq 0 , \lambda \in [0,1]
\}
$
%\end{eqnarray}
and define
$v^1_{1} = [1\ 1]^\top$, $v^1_{2} = [-1\ 1]^\top$,
$v^2_{1} = [1\ -1]^\top$, $v^2_{2} = [-1\ -1]^\top$.
Let $\omega > 0$ and consider the hybrid system
\begin{eqnarray}\label{eqn:HSex2}
\HS\ \left\{
%\begin{array}{l}
\begin{array}{llll}
\dot{x} &= & f(x,u_c) :=
u_c
\left[
\begin{array}{cc}
0 & \omega \\
-\omega & 0
\end{array}
\right]
x
%\end{array}
&
\qquad
(x,u_c) \in C,
\\
x^+ 
&
=
&
g(x,u_d)
& \hspace{-0.85in}
(x,u_d) \in D,
\end{array}
\right.
\end{eqnarray}
\begin{eqnarray}\nonumber
C &:=& \defset{(x,u_c) \in \reals^2\times\reals}{u_c \in \{-1,1\}, x \in \widehat{C}},\\ \nonumber
\widehat{C} &:=& \overline{\reals^2 \setminus ({\cal W}(v^1_1,v^1_2) \cup {\cal W}(v^2_1,v^2_2))},\\
\nonumber
D & := & \defset{(x,u_d) \in \reals^2 \times \realsgeq}{u_d \geq \gamma |x|, x \in
\partial {{\cal W}(v^2_1,v^2_2)}},
\end{eqnarray}
for each $(x,u_d) \in \reals^2\times\realsgeq$ the jump map $g$ is given by
\begin{eqnarray}\nonumber
g(x,u_d) &:=& R(\pi/4) \matt{0 \\ u_d},
%g(x,u_d) &:=& \left\{ R(\pi/4) \matt{0 \\ u_d} ,R(-\pi/4) \matt{0 \\ u_d}\right\},
\quad R(s) = \matt{\cos s & \sin s\\ -\sin s & \cos s},
\end{eqnarray}
and $\gamma > 0$
is such that
$\exp(\pi/(2\omega)) \gamma^2<1$.
For each $i \in \{1,2\}$, the vectors $v^i_1, v^i_2 \in \reals^2$ are 
such that 
${\cal W}(v^1_1,v^1_2) \cap{\cal W}(v^2_1,v^2_2) = \{0\}$.
The set of interest is 
$\A := \{0\} \subset \reals^2$.
Figure~\ref{fig:Ex2FlowJumpSets}
depicts the flow and jump sets projected onto the $x$ plane.
\begin{figure}[h!]  
\begin{center}  
\psfrag{C}[][][0.9]{$C$}
\psfrag{D}[][][0.9]{$D$}
\psfrag{x1}[][][0.9]{\ \ $x_1$}
\psfrag{x2}[][][0.9]{\!\! $x_2$}
\psfrag{W(v11,v21)}[][][0.9]{\ \ \ ${\cal W}(v^1_1,v^1_2)$}
\psfrag{W(v12,v22)}[][][0.9]{\ \ \ ${\cal W}(v^2_1,v^2_2)$}
{\includegraphics[width=.25\textwidth]{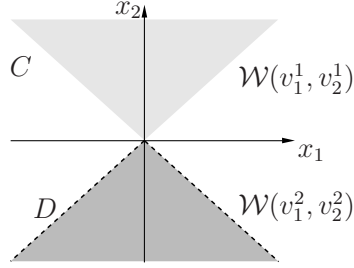}}  
\end{center}  
\caption{Sets for Example~\ref{ex:2}. 
The white region (and its boundary) corresponds to the flow set projected onto the $x$ plane. The dashed line represents $D$.}
\label{fig:Ex2FlowJumpSets}
\end{figure}

To construct a state-feedback law for \eqref{eqn:HSex2},
consider the candidate control Lyapunov function $V$ given by
\begin{equation}\label{eqn:Vex2}
V(x) = \exp(T(x)) x^\top x \qquad \forall x \in \reals^2,
\end{equation}
where $T$ 
denotes the minimum time to reach the 
set ${\cal W}(v^2_1,v^2_2)$ with the continuous dynamics of \eqref{eqn:HSex2} 
and $u_c \in \{-1,1\}$.
The function $T$ is precisely defined as follows.
It is defined as a continuously differentiable function from $\reals^2$ to $[0,\frac{\pi}{2\omega}]$
given as
$T(x) := \frac{1}{\omega}\arcsin\left(\frac{\sqrt{2}}{2}\frac{|x_1|+x_2}{|x|} \right)$
on $\widehat{C}$
and zero for every other point in ${\cal W}(v^2_1,v^2_2)$.
The definition of $V$ is such that 
\eqref{eqn:CLFBounds}
holds with
$\alpha_1(s) := s^2$ and $\alpha_2(s) := \exp\left(\frac{\pi}{2\omega}\right) s^2$ for each $s \geq 0$.

Next, we construct the set-valued maps $\Psi_c$ and $\Psi_d$ and then check \eqref{eqn:CLFFlow} and \eqref{eqn:CLFJump}.
Note that $\Pi(C) = \widehat{C}$
and $\Pi(D) = \partial{{\cal W}(v^2_1,v^2_2)}$.
For each $x \in \reals^2$,
\begin{eqnarray}\non
\Psi_c(x) &=& 
\left\{
\begin{array}{ll}
\{-1,1\} &  \mbox{ if } x \in \widehat{C} \\
\emptyset & \mbox{ otherwise,}
\end{array}
\right.\\ \nonumber
\Psi_d(x) &=& 
\left\{
\begin{array}{ll}
\defset{u_d \in \realsgeq }{u_d \geq \gamma |x|}  & 
\mbox{ if } x \in \partial{{\cal W}(v^2_1,v^2_2),} \\
\emptyset & \mbox{ otherwise}.
\end{array}
\right.
\end{eqnarray}
During flows, we have that
\begin{eqnarray*}
\langle 
\nabla V(x), f(x,u_c)
\rangle
& = & 
\langle 
\nabla T(x), f(x,u_c)
\rangle
V(x)\\
& = &
 \frac{\sqrt{2}}{2\omega}\frac{1}{\sqrt{1 - \frac{1}{2}\left(\frac{|x_1|+x_2}{|x|}\right)^2}}
\left\langle \nabla \frac{|x_1|+x_2}{|x|}, f(x,u_c)\right\rangle V(x) \\
&=&  \frac{u_c}{\omega} \matt{\frac{x_2}{|x|^2} & -\frac{x_1}{|x|^2}}
\left[
\begin{array}{cc}
0 & \omega \\
-\omega & 0
\end{array}
\right]
x 
\,
V(x)
\end{eqnarray*}
for all $(x,u_c) \in C$.  
For $x \in \widehat{C}$, $x_1 > 0$,
$\langle 
\nabla T(x), f(x,u_c)
\rangle= 1$
when $u_c = 1$,
and for $x \in \widehat{C}$, $x_1 < 0$,
$\langle 
\nabla T(x), f(x,u_c)
\rangle=-1$
when $u_c = -1$.
Then
\begin{equation}\label{eqn:FlowDecrease}
\inf_{u_c \in \Psi_c(x)}
\langle 
\nabla V(x), f(x,u_c)
\rangle
 \leq - x^\top x
\end{equation}
 for all 
$x \in \Pi(C)$.
During jumps, we have that, for each $(x,u_d) \in D$,
\begin{eqnarray*}
V(g(x,u_d))  &=&  \exp(T(g(x,u_d))) g(x,u_d)^\top g(x,u_d)\\
& =&  \exp\left(\frac{\pi}{2\omega}\right) u_d^2.
\end{eqnarray*}
It follows that
\begin{eqnarray*}
\non
\inf_{u_d \in \Psi_d(x)} V(g(x,u_d)) - V(x) 
& \leq &
\inf_{u_d \in \Psi_d(x)} \exp\left(\frac{\pi}{2\omega}\right) u_d^2 -  
\exp(T(x)) x^\top x \\ 
& \leq & 
-\left(1- \exp\left(\frac{\pi}{2\omega}\right)\gamma^2\right) x^\top x
\end{eqnarray*}
for each $x \in \Pi(D)$.
Finally, both \eqref{eqn:CLFFlow} and \eqref{eqn:CLFJump} hold
with $s \mapsto \alpha_3(s) := \left(1- \exp\left(\frac{\pi}{2\omega}\right)\gamma^2\right) s^2$.
Then, $V$ is a CLF for \eqref{eqn:HSex2}.

% now select a control law
Now, we determine an asymptotic stabilizing control law for the above hybrid system.
First, we compute the set-valued map ${\cal T}_c$ in \eqref{eqn:calTcAndd}.
To this end, 
the definition of $\Gamma_c$ gives, for each $r \geq 0$,
\begin{eqnarray*}
\Gamma_c(x,u_c,r) \! = \!
\left\{
\begin{array}{ll}\displaystyle
 \frac{u_c}{\omega} \matt{\frac{x_2}{|x|^2} & -\frac{x_1}{|x|^2}}
\left[
\begin{array}{cc}
0 & \omega \\
-\omega & 0
\end{array}
\right]
x 
\,
V(x)
+
\alpha_3(|x|_{\A})
 & 
  \mbox{ if }  (x,u_c) \in C \cap (\Ir\times \reals^{m_c}),\\
-\infty &  \mbox{ otherwise }
\end{array}
\right.
\end{eqnarray*}
from where we get $\Upsilon_c(x,u_c) = \Gamma_c(x,u_c,V(x))$.
Then, 
for each 
$r > 0$ and
$(x,u_c) \in C \cap \left(\Ir \times \reals^{m_c} \right)$,
the set-valued map ${\cal T}_c$ is given by
\begin{eqnarray*}
{\cal T}_c(x) & =  & \Psi_c(x) \cap
\defset{u_c \in \U_c}{ \Upsilon_c(x,u_c) \leq 0} \\
 & = & \{-1,1\} \cap 
\left(
\defset{1}{x_1 > 0}
\cup
\defset{-1}{x_1 < 0}
\right),
\end{eqnarray*}
which reduces to
\begin{eqnarray}\label{eqn:ScMinNormSelection-ex2}
{\cal T}_c(x) = 
\left\{
\begin{array}{ll}
1 & x_1 > 0 \\
-1 & x_1 < 0
\end{array}
\right.
%\Psi_c(x) \cap \defset{u_c \in \U_c}{\Gamma_c(x,u_c)\ \leq\ 0}
\end{eqnarray}
for each $x \in \Pi(C) \cap \defset{x \in \reals^2}{V(x) > 0}$.

Proceeding in the same way,
the definition of $\Gamma_d$ gives, for each $r \geq 0$,
\begin{eqnarray*}
\Gamma_d(x,u_d,r) & = & 
\left\{
\begin{array}{ll}\displaystyle
\exp\left(\frac{\pi}{2\omega}\right) u_d^2 -
 V(x) + {\alpha}_3(|x|_\A) 
 & 
  \mbox{ if } 
    (x,u_d)\in  D \cap (\Ir\times \reals^{m_d}),\\
    \\
-\infty & \mbox{ otherwise}
\end{array}
\right.
\end{eqnarray*}
from where we get $\Upsilon_d(x,u_c) = \Gamma_d(x,u_d,V(x))$.
Then, 
for each 
$r > 0$ and
$(x,u_d) \in D \cap \left(\Ir \times \reals^{m_d} \right)$,
the set-valued map ${\cal T}_d$ is given by
\begin{eqnarray*}
{\cal T}_d(x) & =  & 
\Psi_d(x) \cap \defset{u_d \in \U_d}{\Upsilon_d(x,u_d) \leq 0} \\
& = & 
\defset{u_d \in \realsgeq }{u_d \geq \gamma |x|} \cap 
\defset{u_d \in \realsgeq}{\exp\left(\frac{\pi}{2\omega}\right) u_d^2 -
\exp(T(x)) x^\top x  + {\alpha}_3(|x|_\A) \leq 0}\\
& = & 
\defset{u_d \in \realsgeq }{u_d \geq \gamma |x|}  \cap 
\defset{u_d \in \realsgeq}{\exp\left(\frac{\pi}{2\omega}\right) u_d^2 -
 x^\top x  + {\alpha}_3(|x|_\A) \leq 0}
\end{eqnarray*}
and using the definition of $\alpha_3$, we get
\begin{eqnarray}\non
{\cal T}_d(x)  &=& 
\defset{u_d \in \realsgeq }{u_d \geq \gamma |x|} \cap 
\defset{u_d \in \realsgeq}{\exp\left(\frac{\pi}{2\omega}\right) u_d^2 
- \exp\left(\frac{\pi}{2\omega}\right)\gamma^2 |x|^2
 \leq 0}
\\ \non
& =  &
\defset{u_d \in \realsgeq }{u_d \geq \gamma |x|} \cap 
\defset{u_d \in \realsgeq}{-\gamma |x|\leq u_d \leq \gamma |x|}
\\
\label{eqn:SdMinNormSelection-ex2}
& = &
\defset{u_d \in \realsgeq}{u_d = \gamma |x|}
\end{eqnarray}
for each $x \in \Pi(D) \cap \defset{x \in \reals^2}{V(x) > 0}$.
Then, according to \eqref{eqn:mc}, from \eqref{eqn:ScMinNormSelection-ex2},
for each $x \in \Pi(C) \cap \defset{x \in \reals^2}{V(x) > 0}$ 
we can take the pointwise minimum norm control selection
$$
\rho_c(x) := \left\{
\begin{array}{ll}
1 & x_1 > 0 \\
-1 & x_1 < 0
\end{array}
\right.
$$
According to \eqref{eqn:md}, from \eqref{eqn:SdMinNormSelection-ex2},
for each 
$x \in \Pi(D) \cap \defset{x \in \reals^2}{V(x) > 0}$
we can take the pointwise minimum norm control selection
$$
\rho_d(x) := \gamma |x|.
$$
Figure~\ref{fig:Sim1-ex2} depicts a closed-loop trajectory with 
the control selections above when the region of operation
is restricted to $\defset{x \in \reals^2}{V(x) \geq r}$,
$r = 0.15$.
\begin{figure}[h!]  
\begin{center}  
\psfrag{C}[][][0.9]{$C$}
\psfrag{D}[][][0.9]{$D$}
\psfrag{y=-x}[][][0.9]{}
\psfrag{x1}[][][0.9]{\ \ $x_1$}
\psfrag{x2}[][][0.9][-90]{\!\! $x_2$}
\psfrag{W(v11,v21)}[][][0.9]{\ \ \ ${\cal W}(v^1_1,v^1_2)$}
\psfrag{W(v12,v22)}[][][0.9]{\ \ \ ${\cal W}(v^2_1,v^2_2)$}
{\includegraphics[width=.6\textwidth]{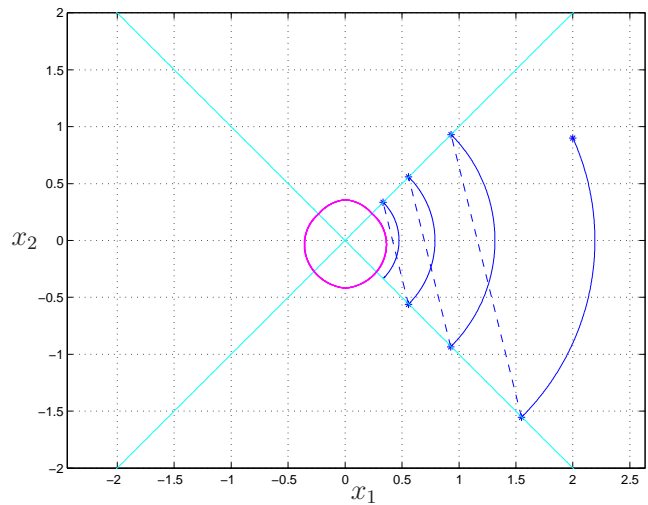}}  
\end{center}  
\caption{Closed-loop trajectory to the system in Example~\ref{ex:2} starting from $x(0,0) = (2,0.9)$ and evolving within $\defset{x \in \reals^2}{V(x) \geq r}$, $r = 0.15$. 
The lines at $\pm 45 \mbox{ deg}$ define the boundary of the flow and jump sets projected onto the $x$ plane.
The $r$-contour plot of $V$ is also shown.}
\label{fig:Sim1-ex2}
\end{figure}
\end{example}

\begin{example}[Impact control of a pendulum]
\label{ex:1}
Consider a point-mass pendulum impacting on a controlled slanted surface.
Denote the pendulum's angle (with respect to the vertical) by $x_1$ and 
the pendulum's velocity (positive when the pendulum rotates in the clockwise direction) by $x_2$.
When $x_1 \geq \mu$ with $\mu$ denoting the angle of the surface,
its continuous evolution is given by
\begin{eqnarray*}
\dot{x}_1 = x_2, \quad 
\dot{x}_2 = -a \sin x_1 - b x_2 + \tau,
\end{eqnarray*}
where $a > 0$, $b \geq 0$ capture the system constants (e.g., gravity, mass, length, and friction) and
$\tau$ corresponds to torque actuation at the pendulum's end.
For simplicity, we assume that $x_1 \in [-\frac{\pi}{2},\pi]$ 
and $\mu \in [-\frac{\pi}{2},0]$.
Impacts between the pendulum and the surface
occur when
\begin{equation}\label{eqn:ImpactSetPendulum}
x_1 \leq\  \mu, \quad x_2 \leq 0.
\end{equation}
At such events, the jump map takes the form
\begin{eqnarray*}
x_1^+ = x_1 + \widetilde{\rho}(\mu) x_1,\qquad
x_2^+ = - e(\mu) x_2,
\end{eqnarray*}
where the functions $\widetilde{\rho}:[-\pi/2,0]\to(-1,0)$ and $e:[-\pi/2,0]\to[0,1)$ are continuous
and capture the effect of pendulum compression and restitution at impacts, respectively,
as a function of $\mu$.
The function $\widetilde{\rho}$ captures rapid displacements of the pendulum at collisions
while $e$ models the effect of the angle $\mu$ on energy dissipation at impacts.
For a vertical surface ($\mu = 0$), $\widetilde{\rho}$ is chosen such that $\widetilde{\rho}(0) \in (-1,0)$
and $e$ is chosen to satisfy
$e(0) = e_0$, where $e_0 \in (0,1)$ is the nominal 
(no gravity effect) restitution coefficient.
For slanted surfaces ($\mu \in [-\frac{\pi}{2},0)$),
when conditions
\eqref{eqn:ImpactSetPendulum} hold,
$\widetilde{\rho}$ is chosen as $x_1 + \widetilde{\rho}(\mu) x_1 > x_1$, $\widetilde{\rho}(\mu) \in (-1,0)$, 
so that, after
the impacts, the pendulum is pushed away from the contact condition,
while the function $e$ is chosen as a nondecreasing function of $\mu$ satisfying $e_0 \leq e(\mu) < 1$ at such angles
so that, due to the effect of the gravity force at impacts, less energy is
dissipated as $|\mu|$ increases.

The model above can be captured by the hybrid system $\HS$ given by
\begin{eqnarray}\label{eqn:HSex1}
\HS\ \left\{
\begin{array}{l}
\left.
\begin{array}{llllll}
\dot{x}_1 & = & x_2 \\
\dot{x}_2 & = & -a \sin x_1 - b x_2 + u_{c,1}
\end{array}
\right\} =: f(x,u_c) \\
\hspace{2in}
\qquad
(x,u_c) \in C,
\\
\left.
\begin{array}{llllll}
x_1^+ & = & x_1 + \widetilde{\rho}(u_{d}) x_1 \\
x_2^+ & = & - e(u_d) x_2
\end{array}\ 
\right\} =: g(x,u_d) \\
\hspace{2in}
\qquad
(x,u_d) \in D,
\end{array}
\right.
\end{eqnarray}
where 
$u_c = [u_{c,1}\ u_{c,2}]^\top = [\tau\ \mu]^\top \in \reals\times[-\frac{\pi}{2},0]=: \U_c$,
$u_d = \mu \in [-\frac{\pi}{2},0] =: \U_d$,
$$
C := \defset{(x,u_c) \in  \left[-\frac{\pi}{2},\pi\right]  \times \reals \times\U_c}{x_1\geq u_{c,2}},$$
$$D := \defset{(x,u_d) \in \left[-\frac{\pi}{2},\pi\right]  \times \reals\times\U_d}{x_1 \leq  u_d, x_2 \leq 0}.$$
Note that the definitions of $C$ and $D$ impose state constraints on the inputs.

Let $\A = \{(0,0)\}$ and consider the candidate control Lyapunov function 
with $\U$ controls for $\HS$ given by
\begin{equation}\label{eqn:CLFpendulum}
V(x) = x^\top P x, \qquad P = \matt{2 & 1 \\ 1 & 1}.
\end{equation}
During flows, we have that
\begin{eqnarray}\non
\langle 
\nabla V(x), f(x,u_c)
\rangle
& = &
4 x_1 x_2 + 2 x_2^2 \\
\non & & \hspace{-0.5in}
+
2 (-a \sin x_1 - b x_2 + u_{c,1}) (x_2 + x_1)
\end{eqnarray}
for all $(x,u_c) \in C$.  It follows that \eqref{eqn:CLFFlow}
is satisfied with $\alpha_3$ defined as $\alpha_3(s) := s^2$ for all $s \geq 0$. 
In fact, 
note that, for each $x \in \reals^2$,
\begin{equation}\non
\Psi_c(x) = 
\left\{
\begin{array}{ll}
\defset{u_c }{x_1 \geq u_{c,2}} = \reals \times [-\frac{\pi}{2},\min\left\{x_{1},0\right\}] & x_1 \in [-\frac{\pi}{2},\pi] \\
\emptyset & x_1 \not \in [-\frac{\pi}{2},\pi].
\end{array}
\right.
\end{equation}
and that
$\Pi(C) = [-\frac{\pi}{2},\pi] \times \reals$.
Then
\begin{equation}\non
\inf_{u_c \in \Psi_c(x)}
\langle 
\nabla V(x), f(x,u_c)
\rangle
 = - x^\top x
\end{equation}
 for all 
$x \in \Pi(C)$ such that $x_1 + x_2 = 0$,
while when $x_1 + x_2 \not= 0$, we have
$$\inf_{u_c \in \Psi_c(x)}
\langle 
\nabla V(x), f(x,u_c)
\rangle = -\infty.
$$
Note that, for each $x\in\reals^2$, we have 
$$
\Psi_d(x) = 
\left\{
\begin{array}{ll}
\defset{u_d }{x_1 \leq  u_{d}} = [x_1,0]\  & x_1 \in [-\frac{\pi}{2},0], x_2 \leq 0 \\
\emptyset & \mbox{ otherwise },
\end{array}
\right.
$$
and that $\Pi(D) = [-\frac{\pi}{2},0]\times (-\infty,0]$.
%Assume that
%%, for each $u_d \in \Psi(x)$, $x \in \Pi(D)$, 
%$\widetilde{\rho}$  and $e$ satisfy
%\begin{eqnarray*}
%\widetilde{\rho}(u_d), e(u_d) \in (0,1)
%\qquad 
%\forall u_d \in \Psi_d(x), x \in \Pi(D). 
%\end{eqnarray*}
Then, during jumps, we have
\begin{eqnarray*}
\inf_{u_d \in \Psi_d(x)}
V(g(x,u_d))  -  V(x) \!\! & = \!\! &  V(g(x,x_1))  -  V(x)\\
%(e_0+e_1)^2 x_2^2 - V(x)\\
&   & 
\hspace{-1.2in} 
\leq
-\min\{
2 (1- (1+\widetilde{\rho}(x_1))^2),
1-e^2(x_1)
\}
x^\top x
\end{eqnarray*}
for all $x \in \Pi(D)$. Then, 
%since $e_0 + e_1 \in [0,\frac{1}{2})$,
condition
\eqref{eqn:CLFJump} is  satisfied
with 
$\alpha_3$ defined as $\alpha_3(s) := \lambda s^2$ for all $s \geq 0$,
$\lambda := \min_{x_1 \in [-\frac{\pi}{2},0]}\{
2 (1- (1+\widetilde{\rho}(x_1))^2),
1-e^2(x_1)
\}$.
It follows that both \eqref{eqn:CLFFlow} and \eqref{eqn:CLFJump} hold
with this choice of $\alpha_3$.

The definition of $\Gamma_c$ gives, for each $r \geq 0$,
\begin{eqnarray*}
\Gamma_c(x,u_c,r) & = & 
\left\{
\begin{array}{ll}\displaystyle
4 x_1 x_2 + 2 x_2^2
+
2 (-a \sin x_1 - b x_2 + u_{c,1}) (x_2 + x_1)
+
\alpha_3(|x|_{\A}) & \\
 & 
\hspace{-1.35in}  \mbox{ if }  (x,u_c) \in C \cap (\Ir\times \reals^{m_c})\\
-\infty &  \mbox{ otherwise }
\end{array}
\right.
\end{eqnarray*}
from where we get $\Upsilon_c(x,u_c) = \Gamma_c(x,u_c,V(x))$.
Then,
for each 
$r > 0$ and
$(x,u_c) \in C \cap \left(\Ir \times \reals^{m_c} \right)$,
the set-valued map ${\cal T}_c$ is given by
\begin{eqnarray}\non
{\cal T}_c(x) &=& \Psi_c(x) \cap
\defset{u_c \in \U_c}{ \Upsilon_c(x,u_c) \leq 0} \\
\non
&=&
\left(\reals \times \left[-\frac{\pi}{2},\min\left\{x_{1},0\right\}\right] \right) \\
\non
& & \hspace{0.3in} \cap
\defset{u_c \in \U_c}{4 x_1 x_2 + 2 x_2^2+2 (-a \sin x_1 - b x_2 + u_{c,1}) (x_2 + x_1)+\alpha_3(|x|_{\A}) \leq 0} \\ \non
&  & \hspace{-0.6in} = 
\defset{u_c \in  \reals \times \left[-\frac{\pi}{2},\min\{x_1,0\}\right]}{4x_1 x_2 + 2 x_2^2 
+ 2(-a \sin x_1 - b x_2 + u_{c,1})  (x_2 + x_1) + \lambda x^\top x \leq 0}
\\
\label{eqn:ScMinNormSelection-ex1}
%\Psi_c(x) \cap \defset{u_c \in \U_c}{\Gamma_c(x,u_c)\ \leq\ 0}
\end{eqnarray}
for each $x \in \Pi(C) \cap \defset{x \in \reals^2}{V(x) > 0}$.
%\IfConf{The constant $\lambda$ is defined as
%$\lambda := \min_{x_1 \in [-\frac{\pi}{2},0]}\{
%2 (1- (1+\widetilde{\rho}(x_1))^2),
%1-e^2(x_1)
%\}$.
%}{}
Proceeding in the same way,
the definition of $\Gamma_d$ gives, for each $r \geq 0$,
\begin{eqnarray*}
\Gamma_d(x,u_d,r) & = & 
\left\{
\begin{array}{ll}\displaystyle
-2 x_1^2 (1-(1+\widetilde{\rho}(u_d))^2) - x_2^2 (1-e^2(u_d)) - 2 x_1 x_2 (2+\widetilde{\rho}(u_d)) e(u_d)
+ {\alpha}_3(|x|_\A) 
 &  \\
 & \hspace{-2.2in}
  \mbox{ if } 
    (x,u_d)\in  D \cap (\Ir\times \reals^{m_d})\\
-\infty & \hspace{-2.2in} \mbox{ otherwise}
\end{array}
\right.
\end{eqnarray*}
from where we get $\Upsilon_d(x,u_c) = \Gamma_d(x,u_d,V(x))$.
Then,
for each 
$r > 0$ and
$(x,u_d) \in D \cap \left(\Ir \times \reals^{m_d} \right)$,
the set-valued map ${\cal T}_d$ is given by
\begin{eqnarray}\non
{\cal T}_d(x) & =  & 
\Psi_d(x) \cap \defset{u_d \in \U_d}{\Upsilon_d(x,u_d) \leq 0} \\
\non
& = & 
%\defset{u_d \in \left[-\frac{\pi}{2},0\right]}{x_1 = u_d}\\
\defset{u_{d} \in \left[-\frac{\pi}{2},0\right] }{ u_{d} \in [x_1,0] }
\\
\non
& & \hspace{-1in} \cap
\defset{u_d \in \reals}{-2 x_1^2 (1-(1+\widetilde{\rho}(u_d))^2) - x_2^2 (1-e^2(u_d)) - 2 x_1 x_2 (2+\widetilde{\rho}(u_d)) e(u_d) + \lambda x^\top x \leq 0}\\
\label{eqn:SdMinNormSelection-ex1}
& = & 
\defset{u_d \in \left[-\frac{\pi}{2},0\right]}{-2 x_1^2 (1-(1+\widetilde{\rho}(u_d))^2) 
- x_2^2 (1-e^2(u_d)) + \lambda x^\top x \leq 0}
%\\
%\label{eqn:SdMinNormSelection-ex1}
%& = & \defset{u_d \in \left[-\frac{\pi}{2},0\right]}{x_1 {\color{red}\ \leq\ } u_d}\ =\ 
%{\color{red}[x_1,0]}
\end{eqnarray}
where we dropped the term
$- 2 x_1 x_2 (2+\widetilde{\rho}(u_d)) e(u_d)$
since on $D$ we have that $x_1 x_2 \geq 0$.

% now select a control law
Defining
$\psi_0(x) :=  4x_1 x_2 + 2 x_2^2 + 2(-a \sin x_1 - b x_2 )  (x_2 + x_1) + \lambda x^\top x$, 
and $\psi_1(x) := 2 (x_1 + x_2)$,
the
\eqref{eqn:ScMinNormSelection-ex1}
can be rewritten as
\begin{eqnarray}\non
{\cal T}_c(x) &=&  
\defset{u_c \in  \reals \times \left[-\frac{\pi}{2},\min\{x_1,0\}\right]}{\psi_0(x) + \psi_1(x) u_{c,1} \leq 0}
\end{eqnarray}
for each $x \in \Pi(C) \cap \defset{x \in \reals^2}{V(x) > 0}$.
To determine the pointwise  minimum norm control selection according to \eqref{eqn:mc},
note that, when $\psi_0(x) \leq 0$, then the pointwise  minimum norm control selection is $u_{c,1}=0$ 
and that, when $\psi_0(x) > 0$, is given by 
$$
-\frac{\psi_0(x) \psi_1(x)}{\psi_1^2(x)} = -\frac{\psi_0(x)}{\psi_1(x)}
$$
which leads to $\psi_0(x) + \psi_1(x) u_{c,1} = 0$.
Then,
the pointwise  minimum norm control selection is given by
$$
\rho_{c,1}(x) := 
\left\{
\begin{array}{ll}
-\frac{\psi_0(x)}{\psi_1(x)} & \psi_0(x) > 0 \\
0 &  \psi_0(x) \leq 0
\end{array}
\right.
\qquad 
\rho_{c,2}(x) :=  0
$$
on $\Pi(C) \cap \defset{x \in \reals^2}{V(x) > 0}$ (see \cite[Chapter 4]{FreemanKokotovic96}).
%where
%$$
%\psi_0(x) = 4 x_1 x_2 + 2 x_2^2 + 2(-a \sin x_1 - b x_2) (x_1 + x_2), \qquad
%\psi_1(x) = 2 (x_1 + x_2).
%$$
According to \eqref{eqn:md}, from \eqref{eqn:SdMinNormSelection-ex1}, since $\widetilde{\rho}$ maps to $(-1,0)$ and $e$ to $(0,1)$,
for each 
$x \in \Pi(D) \cap \defset{x \in \reals^2}{V(x) > 0}$,
the pointwise  minimum norm control selection is given by
$$
\rho_d(x) :=  0.
$$
Since $\rho_{c,2}  = \rho_d$, the selection above can be implemented.

Figure~\ref{fig:Sim1-ex1Planar} depicts a closed-loop trajectory on the plane with
the control selections above when the region of operation
is restricted to $\defset{x \in \reals^2}{V(x) \geq r}$,
$r = 0.0015$. 
Figure~\ref{fig:Sim1-ex1PositionAndVelocity}
shows the position and velocity trajectories projected on the $t$ axis.
The functions $\widetilde{\rho}$ and $e$ used in the simulations
are defined as
$\widetilde{\rho}(s) = 0.5 s - 0.1$
and 
$e(s) = -0.28 s + 0.5$
for each $s \in [-\pi/2,0]$.

\begin{figure}[h!]
\begin{center}
\psfrag{C}[][][0.9]{$C$}
\psfrag{D}[][][0.9]{$D$}
\psfrag{y=-x}[][][0.9]{}
\psfrag{x1}[][][0.9]{\ \ $x_1$}
\psfrag{x2}[][][0.9][-90]{\!\! $x_2$}
\psfrag{W(v11,v21)}[][][0.9]{\ \ \ ${\cal W}(v^1_1,v^1_2)$}
\psfrag{W(v12,v22)}[][][0.9]{\ \ \ ${\cal W}(v^2_1,v^2_2)$}
{\includegraphics[width=.6\textwidth]{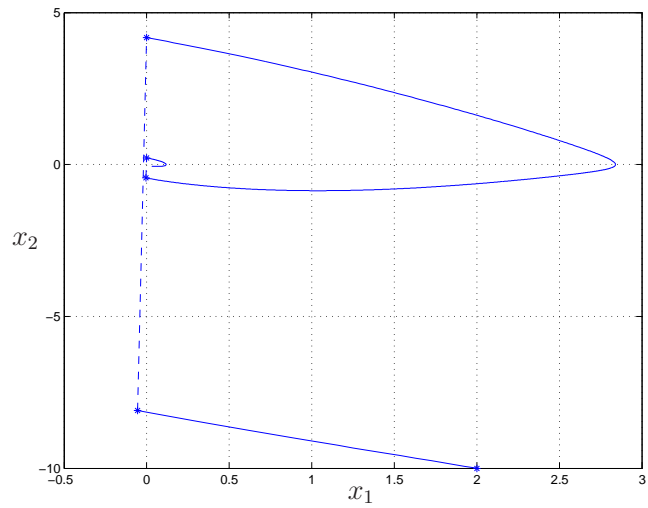}}  
\end{center}  
\caption{Closed-loop trajectory to the system in Example~\ref{ex:1} 
on the plane starting from $x(0,0) = (2,-10)$ and evolving within $\defset{x \in \reals^2}{V(x) \geq r}$, 
$r = 0.0015$.}
\label{fig:Sim1-ex1Planar}
\end{figure}

\begin{figure}[h!]  
\begin{center}  
\psfrag{C}[][][0.9]{$C$}
\psfrag{D}[][][0.9]{$D$}
\psfrag{flows [t]}[][][0.7]{$t [sec]$}
\psfrag{x1}[][][0.9][-90]{\ \ $x_1$}
\psfrag{x2}[][][0.9][-90]{\!\! $x_2$}
\psfrag{W(v11,v21)}[][][0.9]{\ \ \ ${\cal W}(v^1_1,v^1_2)$}
\psfrag{W(v12,v22)}[][][0.9]{\ \ \ ${\cal W}(v^2_1,v^2_2)$}
{\includegraphics[width=.6\textwidth]{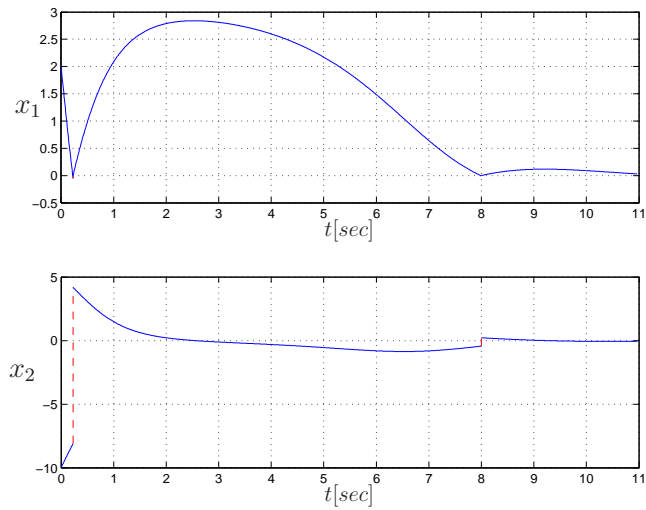}}  
\end{center}  
\caption{Closed-loop position ($x_1$) 
and velocity ($x_2$) to the system in Example~\ref{ex:1} starting from $x(0,0) = (2,-10)$ and evolving within $\defset{x \in \reals^2}{V(x) \geq r}$, $r = 0.0015$.}
\label{fig:Sim1-ex1PositionAndVelocity}
\end{figure}
\end{example}

\begin{example}[Desynchronization of coupled timers with controlled resets]
\label{ex:3}
% FF example
Consider the hybrid system 
with state
$$
x:= \matt{ \ton \\ \ttw } \in P := [0,\tb] \times [0,\tb],
$$
with 
 $x_{1}, x_2$ being timer states with threshold $\tb >0$.
 The state $x$
evolves continuously according to the flow map
$$
f(x) := \matt{ 1 \\ 1}
$$
when
\begin{equation}
x \in C := P
\label{eqn:flowset}
\end{equation}
The state $x$ jumps when any of the timers expires.
Defining inputs affecting the jumps by 
$u_d = (u_{d,1},u_{d,2}) \in P$,
jumps will be triggered when
\begin{equation}
(x,u_d) \in D := \{ (x,u_d) \in P \times P: \mbox{max}\{\ton,\ttw\} = \tb \} . 
\label{eqn:D}
\end{equation}
At jumps, if a timer $x_i$ reached the threshold $\tb$,
then it
gets reset to the value of the respective input component of $u_{d,i}$,
while if $x_j$, $j \not = i$, did not reach the threshold
then it gets reduced by a fraction of its value.
More precisely, the jump map is given by
$$
g(x,u_d) = \left[ \begin{array}{c} g(x_1,x_2,u_{d,1}) \\ g(x_2,x_1,u_{d,2}) \end{array} \right] \qquad \forall (x,u_d) \in D,
$$ 
where $g$ is defined as
%$$
%g(s_1,s_2) = \left\{ \begin{array}{ll} (1+\eps)s_1 & \mbox{if } s_1 < \tb \\ s_2 & \mbox{if } s_1 > \tb \\ \{(1+\eps)s_1, s_2\} & \mbox{if } s_1 = \tb \end{array} \right.
%$$
$$
g(s_1,s_2,s_3) = \left\{ \begin{array}{ll} (1+\eps)s_1 & \mbox{if } s_1 < \tb, s_2 = \tb \\ s_3 & \mbox{if } s_1 = \tb, s_2 < \tb \\ \{(1+\eps)s_1, s_3\} & \mbox{if } s_1 = \tb, s_2 = \tb \end{array} \right.
\qquad \forall (s_1,s_2) \in \Pi(D), s_3 \in P
$$
with parameter $\varepsilon \in (-1,0)$.

We are interested in the asymptotic stabilization of the set
\begin{equation}\label{eqn:A-ex3}
\A := \defset{x \in P}{ |x_2 - x_1 | = k}, \qquad k >0,
\end{equation}
which, 
for an appropriate $k$, 
would correspond
to the two timers being desynchronized since
asymptotic stability of $\A$ would imply
$$
\lim_{(t,j)\in \dom x, \ t+j \to \infty }|x_2(t,j) - x_1(t,j)| = k > 0
$$
for every complete solution $x$.
Let
$k = \frac{\eps + 1}{\eps + 2}\tb$, which for $\eps \in (-1,0)$ is such that $k \in (0,\tb)$.
Consider the candidate control Lyapunov function $V:P \to \reals$ given by
\begin{equation}\label{eqn:Vex3}
V(x) = \min
\left\{
\left|
x_2 - x_1 + k
\right|,
\left|
x_2 - x_1 - k
\right|
\right\}
\end{equation}  
Defining
$$\widetilde{\A} = \widetilde{\ell}_1 \cup \widetilde{\ell}_2 \supset \A,$$
where 
%$\widetilde{\ell}_2$ are extensions of $\ell_{1}$ and $\ell_{2}$ given by
\begin{align}
\begin{split}
\widetilde{\ell}_1& = \{x  : \left[\begin{array}{c} \tb \\ \frac{\tb}{\varepsilon + 2} \end{array}\right] + \one t \in P \cup \sqrt{2}\tb\mathbb{B}, t\in \reals \} ,\\ 
\widetilde{\ell}_2& = \{x : \left[\begin{array}{c}  \frac{\tb}{\varepsilon + 2} \\ \tb \end{array}\right] + \one t \in P \cup \sqrt{2}\tb\mathbb{B}, t\in \reals \} .
\end{split}
\label{eqn:ellH2}
\end{align}
Note 
that $\widetilde{\A}_2$ is an inflation of $\A_2$ and is such that
$V(x) = |x|_{\widetilde{\A}_2}$ on $P$.

Next, we construct the set-valued map $\Psi_d$, 
and then check \eqref{eqn:CLFJump}.
Note that $\Pi(D) = \defset{x}{\max\{ x_1,x_2\}=\tb}$.
For each $x \in \reals^2$,
\begin{equation}\non
\Psi_d(x) = \left\{
\begin{array}{ll}
P &  \mbox{ if } x \in \Pi(D) \\
\emptyset & \mbox{ otherwise,}
\end{array}
\right.\ \ 
\end{equation}
We have the following properties.
For all $x \in C$
where $V$ is differentiable, we obtain
\begin{eqnarray}
\langle 
\nabla V(x), f(x)
\rangle
 = 0
\end{eqnarray}
For each $(x,u_d) \in D$, we have that there exists
$i \in \{1,2\}$ such that
$x_i = \tb$ and $x_j \leq \tb$.
Without loss of generality, suppose
that $i = 1$ and $j = 2$.
Then,
$\eta \in g(x,u_d)$,
is such that
$\eta_1 = u_{d,1}$
and $\eta_2 = (1+\varepsilon)x_2$
if $x_2 < \tb$,
while 
$\eta_1 \in \{(1+\varepsilon)x_1,u_{d,1}\}$
and $\eta_2 \in \{(1+\varepsilon)x_2,u_{d,2}\}$
if $x_2 = \tb$.
Then,
for each $(x,u_d) \in D$,
\begin{eqnarray*}
V(\eta) - V(x) & =&  
 \min
\left\{
\left|
\eta_2 - \eta_1 + k
\right|,
\left|
\eta_2 - \eta_1 - k
\right|
\right\}
-
 \min
\left\{
\left|
x_2 - \tb + k
\right|,
\left|
x_2 - \tb - k
\right|
\right\}
\\
& = &
 \min
\left\{
\left|
\eta_2 - \eta_1 + k
\right|,
\left|
\eta_2 - \eta_1 - k
\right|
\right\}
 -
\left| x_2  - \tb +  k \right|.
\end{eqnarray*}
Using the fact that
 $k = \frac{1+\eps}{2+\eps}\tb$,
it follows that for every $x \in \Pi(D)$, $x_1 = \tb$, $x_2 \leq \tb$, 
$\eta \in g(x,u_d)$,
we have
\begin{eqnarray}
\inf_{u_d \in \Psi_d(x)} V(\eta) - V(x) \leq \eps \left| x_2 - \frac{\tb}{2 + \eps} \right| 
=\eps \left|
|x_2 - x_1| - k
\right| = \eps 
|x|_{\widetilde{\A}}.
\end{eqnarray}
Proceeding similarly for every other point in $D$,
we have that 
\eqref{eqn:CLFJump} holds
with $s \mapsto \alpha_3(s) := -\eps s$.

% control selection
Now, we determine an asymptotic stabilizing control law for the above hybrid system.
We compute the set-valued map ${\cal T}_d$ in \eqref{eqn:calTcAndd}.
To this end, 
the definition of $\Gamma_d$ gives, for each $r \geq 0$,
\begin{eqnarray*}
\Gamma_d(x,u_d,r) & = & 
\left\{
\begin{array}{ll}\displaystyle
\max_{\eta \in g(x,u_d)} V(\eta) - V(x)
+
\alpha_3(|x|_{\A})
 & 
  \mbox{ if }  (x,u_d) \in D \cap (\Ir\times \reals^{m_d})\\
-\infty &  \mbox{ otherwise }
\end{array}
\right.
% \\
%& = & 
%\left\{
%\begin{array}{ll}\displaystyle
%{\color{red}
%\max_{\eta \in g(x,u_d)} 
% \min
%\left\{
%\left|
%\eta_2 - \eta_1 + k
%\right|,
%\left|
%\eta_2 - \eta_1 - k
%\right|
%\right\}
% -
%\left| x_2  - \tb +  k \right|
%}
%-\eps |x|_{\widetilde{\A}} & \\
%\null
% & 
%\hspace{-2in}  \mbox{ if }  (x,u_d) \in D \cap (\Ir\times \reals^{m_d}),\\
%-\infty & 
%\hspace{-2in}  \mbox{ otherwise }
%\end{array}
%\right.
\end{eqnarray*}
from where we get $\Upsilon_d(x,u_d) = \Gamma_d(x,u_d,V(x))$.
Then, 
for each 
$r > 0$ and
$(x,u_d) \in D \cap \left(\Ir \cap \reals^{m_d} \right)$,
the set-valued map ${\cal T}_d$ is given by
\begin{eqnarray*}
{\cal T}_d(x) & =  & \Psi_d(x) \cap
\defset{u_d \in \U_d}{ \Upsilon_d(x,u_d) \leq 0} \\
& = & \defset{u_d \in P}{  \max_{\eta \in g(x,u_d)} 
V(\eta) - V(x)
-\eps |x|_{\widetilde{\A}} \leq 0}
\end{eqnarray*}

To determine the pointwise minimum norm control $u_d$, 
consider again 
$x_1 = \tb$ and $x_2 \leq \tb$,
which implies that $\eta \in g(x,u_d)$
is such that
$\eta_1 = u_{d,1}$
and $\eta_2 = (1+\varepsilon)x_2$
if $x_2 < \tb$,
while 
$\eta_1 \in \{(1+\varepsilon)x_1,u_{d,1}\}$
and $\eta_2 \in \{(1+\varepsilon)x_2,u_{d,2}\}$
if $x_2 = \tb$.
Then
%\begin{itemize}
%\item 
if $x_2 < \tb$
\begin{eqnarray*}
{\cal T}_d(x) & =  &  
 \defset{u_d \in P}{ 
 \min
\left\{
\left|
\eta_2 - \eta_1 + k
\right|,
\left|
\eta_2 - \eta_1 - k
\right|
\right\}
 -
\left| x_2  - \tb +  k \right|
-\eps |x|_{\widetilde{\A}} \leq 0}
\\
& = & \defset{u_d \in P}{ 
 \min
\left\{
\left|
(1+\varepsilon)x_2 - u_{d,1} + k
\right|,
\left|
(1+\varepsilon)x_2 - u_{d,1} - k
\right|
\right\}
 -
\left|x_2  - \tb +  k \right|
-\eps |x|_{\widetilde{\A}} \leq 0}
\end{eqnarray*}
%
%\begin{eqnarray*}
%{\cal T}_d(x) & =  &  \defset{u_d \in P}{ 
% \min
%\left\{
%\left|
%\eta_2 - \eta_1 + k
%\right|,
%\left|
%\eta_2 - \eta_1 - k
%\right|
%\right\}
% -
%\left| x_2  - \tb +  k \right|
%-\eps |x|_{\widetilde{\A}} \leq 0}
%\end{eqnarray*}
%\end{itemize}
For each $x_2 < \tb$, $(u_{d,1},u_{d,2})$ with $u_{d,1} = 0$  belongs to ${\cal T}_d(x)$
since
\begin{eqnarray}
& &  \min
\left\{
\left|
(1+\varepsilon)x_2 - 0 + k
\right|,
\left|
(1+\varepsilon)x_2 - 0 - k
\right|
\right\}
 -
\left|x_2  - \tb +  k \right|
-\eps |x|_{\widetilde{\A}}  \\
& & =  
\left|
(1+\varepsilon)x_2 - k
\right|
 -
\left|x_2  - \tb +  k \right|
-\eps ||x_2 - x_1| - k| \\
&  & =  
\left|
(1+\varepsilon)x_2 - k
\right|
 - (1+\varepsilon)
\left|x_2  - \tb +  k \right| \\
&  & =  
(1+\varepsilon) 
\left(
\left|
x_2 - \frac{k}{1+\varepsilon}
\right|
 - 
\left|x_2  - \tb +  k \right|
\right) \\
&  & =  
(1+\varepsilon) 
\left(
\left|
x_2 - \frac{k}{1+\varepsilon}
\right|
 - 
\left|x_2  - \tb +  k \right|
\right) \\
& & = 0
\end{eqnarray}
since
$\frac{k}{1+\varepsilon} = \frac{\tb}{2+\varepsilon}$
and 
$- \tb +  k = -\frac{\tb}{2+\varepsilon}$.
When $x_2 = \tb$, then $\eta_1 = u_{d,1}$ and
$\eta_2 = u_{d,2}$ are possible values of $\eta$, in which case
$u_{d,1} = u_{d,2} = 0$ belong to ${\cal T}_d(x)$.
The same property holds for every other possibility of $\eta$.

Then, according to \eqref{eqn:md}, 
for each 
$x \in \Pi(D)$
we can take the pointwise  minimum norm control selection
$$
\rho_d(x) := 0.
$$

%Figure~\ref{fig:Sim1-ex3} depicts a closed-loop trajectory with 
%the control selections above when the region of operation
%is restricted to $\defset{x \in P}{V(x) \geq r}$,
%$r = 0.1$.  
%\begin{figure}[h!]  
%\begin{center}  
%\psfrag{C}[][][0.9]{$C$}
%\psfrag{D}[][][0.9]{$D$}
%\psfrag{x1}[][][0.9]{\ \ $x_1$}
%\psfrag{x2}[][][0.9]{\!\! $x_2$}
%\psfrag{W(v11,v21)}[][][0.9]{\ \ \ ${\cal W}(v^1_1,v^1_2)$}
%\psfrag{W(v12,v22)}[][][0.9]{\ \ \ ${\cal W}(v^2_1,v^2_2)$}
%%{\includegraphics[width=.6\textwidth]{/Users/Ricardo/svn/papers/2011/Sanfelice.11.CDC/Matlab/RotateAndDissipateExample/Planar.eps}}  
%\end{center}  
%\caption{Closed-loop trajectory to the system in Example~\ref{ex:3}. 
%The dashed lines define the boundary  of the flow and jump sets projected onto the $x$ plane.}
%\label{fig:Sim1-ex3}
%\end{figure}
\end{example}

%%%%%%%%%%%%%%%%%%%%%%%%%%%%%%%   NEW   %%%%%%%%%%%%%%%%%%%%%%%%%%%%%%%%%%%%

\vspace{0.1in}
\balance
\bibliographystyle{unsrt}
\bibliography{long,Biblio,RGS}

\end{document}